\documentclass[titlepage,12pt]{article}\usepackage[top=2.5cm,left=2.5cm,right=2.5cm,bottom=2.5cm]{geometry}
\usepackage{amsmath}
\usepackage{amsthm}
\usepackage{amssymb}
\usepackage{amscd}
\usepackage{amsfonts}
\usepackage{amsbsy}
\usepackage{graphicx}
\usepackage[dvips]{psfrag}
\usepackage{setspace}
\usepackage{indentfirst}
\usepackage{rotating,graphics,psfrag,epsfig}
\usepackage{float}
\usepackage{hyperref}
\usepackage{booktabs}
\usepackage{xcolor}
\usepackage{mathrsfs}
\usepackage{overpic}
\usepackage{doi}
\usepackage{enumitem}  
\usepackage{tikz}

\makeatletter
\hypersetup{
		urlbordercolor=0 0 1,
		colorlinks=true,       	
    linkcolor=blue,
    citecolor=blue,
    filecolor=blue,
		urlcolor=blue,
		bookmarksdepth=4
	}

\newtheorem{teo}{Theorem}
\newtheorem{corol}{Corollary}
\newtheorem{obs}{Remark}
\newtheorem{propo}{Proposition}
\newtheorem{lem}{Lemma}
\newtheorem{defi}{Definition}

\newtheorem{main}{Theorem} 

\begin{document}

\onehalfspacing

\begin{center}
{\Large{\bf On the cyclicity of persistent hyperbolic polycycles}}
\end{center}

\bigskip

\begin{center}
{\large Lucas Queiroz Arakaki$^\ast$, Paulo Santana}\\
{\em UNICAMP, Campinas, SP, Brazil.} \\
{\em IBILCE-UNESP, S. J. Rio Preto, SP, Brazil.} \\
e-mail: larakaki@unicamp.br, paulo.santana@unesp.br
\end{center}

\bigskip

\begin{center}
{\bf Abstract}
\end{center}
    In this work we consider families of smooth vector fields having a persistent polycycle with $n$ hyperbolic saddles. We derive the asymptotic expansion of the return map associated to the polycycle, determining explicitly its leading terms. As a consequence, explicit conditions on the leading terms allow us to determine the cyclicity of such polycycles. We then apply our results to study the cyclicity of a polycycle of a model with applications in Game Theory.

\bigskip



\noindent \makebox [40mm]{\hrulefill}

\noindent {\footnotesize $^\ast$Corresponding author}.


\section{Introduction}

As an extensive effort to prove the existential part of Hilbert's sixteenth problem, several authors worked in proving the finite cyclicity of the limit periodic sets inside polynomial vector fields, since the finite cyclicity of the limit periodic sets implies that the number of limit cycles is also finite \cite{RoussarieNote}. A limit periodic set inside a polynomial vector field is one of the following: a singular point, a periodic orbit or a graphic. 

The cyclicity of graphics was extensively studied in the literature (see, for instance \cite{Dumortier_1994, Rousseau98, DeMaesschalck11}). For hyperbolic polycycles, i. e. graphics whose corners are hyperbolic saddles and with a well defined return map on one of its sides, it is essential to understand the behavior and properties of the \emph{Dulac map} which is the transition map in the neighborhood of a hyperbolic saddle. In this regard, the works of A. Mourtada \cite{Mourtada1, Mourtada2, Mourtada3, Mourtada4} have substantially developed the understanding of the Dulac map by obtaining a normal form for the Dulac map, namely
    \[D(s;\mu)=s^{\lambda(\mu)}(A(\mu)+R(s;\mu)),\]
where $r(\mu)$ is the \emph{hyperbolicity ratio} of the hyperbolic saddle and $R(s;\mu)$ is a well-behaved remainder. The Mourtada's normal form was further studied to obtain some results on the stability of a generic polycycle \cite{GasManMan02} and an upper bound on the cyclicity \cite{Panazzolo}. 

Recently Marín and Villadelprat~\cite{MarVilDulacLocal, MarVilDulacGeneral, MarVilDulacCoef} proved several results on the Dulac map, which improved uppon Mourtada's normal form. More precisely, they obtained an asymptotic development of the Dulac map and proved that the remainder $R(s;\mu)$ belongs to a class of finitely flat functions. Using their asymptotic development, several advancements in the study of the cyclicity of hyperbolic polycycles have been made (see \cite{MarVilKolmogorov,SantanaPoly}). 

When dealing with perturbations of hyperbolic polycycles, in the context of bifurcation of limit cycles, the generic behavior is the breaking of one of its saddle connections (see, for instance \cite{SantanaPoly,GasManVil05}). In the non-generic scenario where all saddle connections remain unbroken throughout the perturbation, we say that the polycycle is \emph{persistent}. This type of polycycle was studied in \cite{MarVilKolmogorov, QMV}. In this regard, in \cite{MarVilKolmogorov}, the authors studied the cyclicity of the persistent polycycle with three corners that arise in Kolmogorov systems. Their approach was to study the return map associated to the polycycle and obtaining explicit expressions for its leading terms which define three functions that played the same role for the cyclicity of the polycycle as the Lyapunov quantities' role for the cyclicity of a focus.  

In the present paper, we will consider the cyclicity of persistent polycycles. Our goal is to generalize the results of \cite{MarVilKolmogorov} to a more general class of persistent polycycles. In this direction, we obtain the explicit expressions for the leading terms of the return map under some assumptions which then allow us to state some conditions on these leading terms so that the cyclicity of the polycycle is determined. 

\section{Statement of the main results}

We now provide the necessary definitions for a precise statement of our main results.

\begin{defi}[Polycycle]
    Let $X$ be a two-dimensional vector field. A \emph{graphic} $\Gamma$ for $X$ is a compact non-empty invariant subset which is a continuous image of $\mathbb{S}^1$ and consists of a finite number of (not necessarily distinct) isolated singular points $\{p_1,\dots,p_n,p_{n+1}=p_1\}$ and compatibly oriented separatrices $\{\gamma_1,\dots,\gamma_n\}$ connecting them (meaning that $\gamma_i$ has $\{p_i\}$ as the $\alpha$-limit set and $\{p_{i+1}\}$ as the $\omega$-limit set). A graphic for which all its singular points are hyperbolic saddles is said to be \emph{hyperbolic}. A \emph{polycycle} is a graphic with a well-defined first return map $\mathscr{R}$ on one of its sides, see Figure~\ref{Fig1}.
\end{defi}

\begin{figure}[ht]
	\begin{center}
		\begin{minipage}{8cm}
			\begin{center} 
			 	\begin{overpic}[width=6cm]{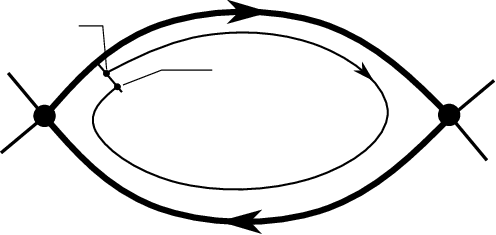} 
				    \put(95,22){$p_1$}
			     	\put(-1,23){$p_2$}
			     	\put(76,40){$L_1$}
			      	\put(15,5){$L_2$}
                        \put(11,41){$s$}
                        \put(43,31){$\mathscr{R}(s)$}
			 	\end{overpic}
				
		  		$(a)$
			\end{center}
		\end{minipage}
		\begin{minipage}{8cm}
			\begin{center} 
			 	\begin{overpic}[width=6cm]{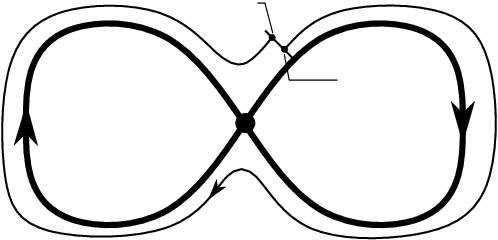} 
			     	\put(54,22.5){$p_1=p_2$}
		      	  \put(74,6){$L_1$}
			     	\put(15,36){$L_2$}
                        \put(69,30){$s$}
                        \put(36,46){$\mathscr{R}(s)$}
		      	\end{overpic}
				
		  		$(b)$
		  	\end{center}
		\end{minipage}
	\end{center}
	\caption{Illustration of $\Gamma$, with $(a)$ distinct and $(b)$ non-distinct hyperbolic saddles.}\label{Fig1}
\end{figure}
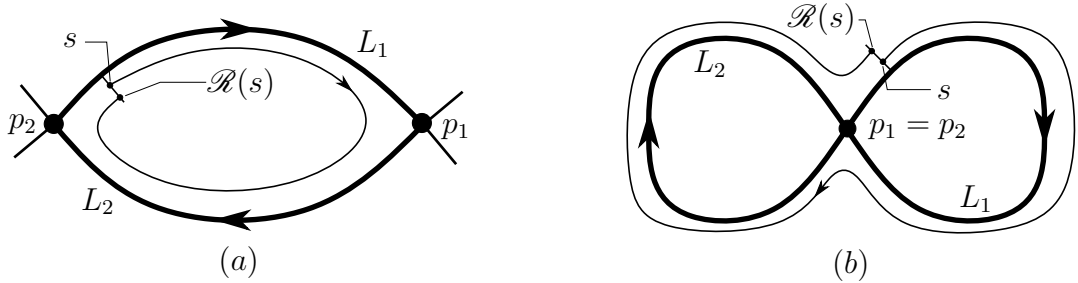
	
\begin{defi}[Persistent polycycle]
    Let $\{X_\mu\}_{\mu\in\Lambda}$ be a smooth (i.e. of class $C^\infty$) family of planar smooth vector fields such that $\Gamma$ is a hyperbolic polycycle of $X_{\mu_0}$. We say that $\Gamma$ is a \emph{persistent polycycle} when all of its separatrix connections remain unbroken inside the family $\{X_\mu\}_{\mu\in\Lambda}$.
\end{defi}

\begin{defi}[Independent functions]\label{Def1}
    Let $\Lambda$ be a topological space and consider a set of functions $f_i:\Lambda\to\mathbb{R}$, $i\in\{1,\dots,m\}$. For each $k\in\{1,\dots,m\}$ we denote by
      \[V(f_1,\dots,f_k)=\{\mu\in\Lambda:f_i(\mu)=0,\,i\in\{1,\dots,k\}\}\]
    the variety defined by $f_1,\dots,f_k$. Given $\mu_0\in V(f_1,\dots,f_m)$, we say that $f_1,\dots,f_m$ \emph{are independent at} $\mu_0$, when the following holds:
\begin{itemize}
    \item If $\mu_0\in V(f_1)$, then every neighborhood of $\mu_0$ contains two points $\mu_1,\mu_2$ such that $f_1(\mu_1)f_1(\mu_2)<0$;
    \item For every $k\in\{2,\dots,m\}$ and every neighborhood $U$ of $\mu_0$, there are two points $\mu_1,\mu_2\in U\cap V(f_1,\dots,f_{k-1})$ such that $f_k(\mu_1)f_k(\mu_2)<0$.
\end{itemize}
\end{defi}

Observe that if $\Lambda\subset\mathbb{R}^N$, $N\geqslant m$, is an open set, the functions $f_1,\dots,f_m$ are of class $C^1$ and the gradients $\nabla f_1(\mu_0),\dots,\nabla f_m(\mu_0)$ are linearly independent vectors of $\mathbb{R}^N$, then there is a neighborhood $U\subset\Lambda$ of $\mu_0$ whose restriction of $f_1,\dots,f_m$ to $U$ are independent at $\mu_0$.

In the next definition $\rm{dist}_H$ stands for the Hausdorff distance between compact sets of $\mathbb{R}^2$.
	
\begin{defi}[Cyclicity]
    Let $\{X_\mu\}_{\mu\in\Lambda}$ be a smooth family of planar smooth vector fields on and suppose that $\Gamma$ is a polycycle for $X_{\mu_0}$. We say that $\Gamma$ has finite cyclicity in the family $\{X_\mu\}_{\mu\in\Lambda}$ if there exist $\kappa\in\mathbb{N}$, $\varepsilon>0$ and $\delta>0$ such that any $X_{\mu}$ with $\vert\mu-\mu_0\vert<\delta$ has at most $\kappa$ limit cycles $\gamma_i$ with $\rm{dist}_H(\Gamma,\gamma_i)<\varepsilon$ for $i\in\{1,\dots,\kappa\}$. The minimum of such $\kappa$ when $\delta$ and $\varepsilon$ go to zero is called \emph{cyclicity} of $\Gamma$ in  $\{X_\mu\}_{\mu\in\Lambda}$ at $\mu=\mu_0$ and denoted by $\rm{Cycl}(\Gamma,\mu_0)$.
\end{defi}

Consider a smooth family $\{X_\mu\}_{\mu\in\Lambda}$ of planar smooth vector fields having a persistent polycycle $\Gamma$ with $m$ hyperbolic saddles $p_1,\dots,p_n$ (to simplify the notation we have dropped the dependence on $\mu$ of the saddles). Let $\lambda_i^s(\mu)<0<\lambda_i^u(\mu)$ be the associated eigenvalues of $p_i$. The \emph{hyperbolicity ratio} of $p_i$ is the positive real number given by
\begin{equation}\label{4}
    \lambda_i(\mu):=\frac{|\lambda_i^s(\mu)|}{\lambda_i^u(\mu)}.
\end{equation}
The product of the hyperbolicity ratios 
\begin{equation}\label{19}
    r(\mu)=\prod_{i=1}^{n}\lambda_i(\mu),
\end{equation}
is called \emph{graphic number} of $\Gamma$. Note that the case $n=1$ corresponds to a saddle loop.

\v Cerkas~\cite{Cherkas} proved that if $r(\mu_0)\neq1$, then $\Gamma$ has a well defined stability. More precisely if $r(\mu_0)>1$, then $\Gamma$ is stable (i.e. it attracts the orbits in the region where the first return map is defined). Similarly if $r(\mu_0)<1$, then $\Gamma$ is unstable. Since $r(\Gamma)$ depends continuously on smooth perturbations, it follows that if $\Gamma$ is persistent and $r(\mu_0)\neq1$, then $\Gamma$ has no change of stability for small perturbations. According with the terminology introduced by Sotomayor~\cite[Section 2.2]{Soto}, if $r(\mu_0)\neq1$ then we say that $\Gamma$ is a \emph{simple} polycycle. 

As anticipated in the Introduction, in recent years Marín and Villadelprat~\cite{MarVilDulacLocal, MarVilDulacGeneral, MarVilDulacCoef} proved several results on the Dulac map of hyperbolic saddles. For simplicity we postpone their precise statements to Section~\ref{Sec:Preliminary}. In what follows we state only a simple version of their results and definitions, sufficient for the statement of our first main result.

\begin{defi}[Well-behaved remainder]
    Consider an open set $U\subset\mathbb{R}^N$ and a smooth function $\psi\colon(0,\varepsilon)\times U\to\mathbb{R}$, with $\varepsilon>0$ small. Given $\ell\in\mathbb{R}$ and $\mu_0\in U$, we write $\psi\in\mathcal{F}^\infty_\ell(\mu_0)$ if for each $\nu=(\nu_0,\nu_1,\dots,\nu_n)\in\mathbb{Z}_{\geqslant0}^{N+1}$ there are a neighborhood $V\subset U$ of $\mu_0$, $C>0$ and $s_0>0$ such that
     \[\left\vert\dfrac{\partial^{|\nu|}\psi}{\partial s^{\nu_0}\partial\mu_1^{\nu_1}\cdots\partial\mu_N^{\nu_N}}(s;\mu)\right\vert\leqslant Cs^{\ell-\nu_0}\;\text{for all } s\in(0,s_0)\text{ and } \mu\in V,\]
     where $|\nu|=\nu_0+\dots+\nu_N$ and $\mu=(\mu_1,\dots,\mu_N)$.
\end{defi}

From~\cite{MarVilDulacLocal,MarVilDulacGeneral,MarVilDulacCoef} we have that the Dulac map of the hyperbolic saddle $p_i$ can be written as
\begin{equation}\label{16}
    D_i(s;\mu)=s^{\lambda_i}\bigl(\Delta_{00}^i+\mathcal{F}^\infty_\ell(\mu_0)\bigr),
\end{equation}
for any given $\ell\in(0,\min\{\lambda_i^0,1\})$, where $\lambda_i=\lambda_i(\mu)$ is the hyperbolicity ratio~\eqref{4}, $\lambda_i^0=\lambda_i(\mu_0)$, and $\Delta_{00}^i=\Delta_{00}^i(\mu)$ is a strictly positive smooth function defined in a neighborhood of $\mu_0$ (we dropped the $\mu$-dependence at the right-hand side of \eqref{16} for simplicity). For the explicit expression of $\Delta_{00}^i$, see Appendix~\ref{App:Dulac}. Given $j$, $k\in\{0,\dots,n\}$, $j\leqslant k$, we define

\begin{equation}\label{eq:formulasAlambda}
A_{j,k}=\prod_{i=j}^{k}(\Delta_{00}^{i})^{\Lambda_{i,k}}, \quad \Lambda_{i,k}=\prod_{j=i+1}^{k}\lambda_j, \quad \Lambda_{kk}=1, \quad \Lambda_{i,k}^0=\Lambda_{i,k}(\mu_0).
\end{equation}
    
In our first main result we provide an explicit expression for the first return map of a persistent polycycle and use this expression to study its cyclicity. We recall that $r(\mu)$ denotes the graphic number~\eqref{19} of $\Gamma$.

\begin{main}\label{Teo:Return0Cycl012}
Let $\{X_\mu\}_{\mu\in\Lambda}$ be a smooth family of planar analytic vector fields having a persistent polycycle $\Gamma$ with hyperbolic saddles $p_1,\dots,p_n$. Then the first return map associated to $\Gamma$ is given by
\begin{equation}\label{eq:ReturnmapL0}
    \mathscr{R}(s;\mu)=s^{r(\mu)}\bigl(A_{1,n}+\mathcal{F}_\ell^\infty(\mu_0)\bigr),
\end{equation}
for any given $\ell\in\bigl(0,\min\bigl\{\Lambda_{i,n}\colon i\in\{0,\dots,n\}\bigr\}\bigr)$. Moreover, the following holds:
\begin{itemize}
    \item [(a)] ${\rm Cycl}(\Gamma,\mu_0)=0$, if $r(\mu_0)\neq 1$;
    \item [(b)] ${\rm Cycl}(\Gamma,\mu_0)\geqslant 1$, if $r(\mu_0)=1$, $r(\mu)-1$ changes signs at $\mu_0$ and $\mathscr{R}(\cdot;\mu_0)\not\equiv Id$;
    \item [(c)] ${\rm Cycl}(\Gamma,\mu_0)\leqslant 1$, if $A_{1,n}(\mu_0)\neq 1$;
    \item [(d)] ${\rm Cycl}(\Gamma,\mu_0)\geqslant 2$, if  $r(\mu_0)=A_{1,n}(\mu_0)=1$, $r-1,\;A_{1,n}-1$ are independent at $\mu_0$ and $\mathscr{R}(\cdot;\mu_0)\not\equiv Id$.  
\end{itemize}
\end{main}

We observe that the expression~\eqref{eq:ReturnmapL0} of the first return map is similar to the expressions~\eqref{16} of the Dulac maps. Hence we say that it is of \emph{Dulac-type}. 

In order to obtain conditions for a higher cyclicity, it is necessary to obtain a more refined expression for the first return map. This in turn implies in the necessity to study more refined expressions of the Dulac maps. To this end, we briefly observe that from~\cite{MarVilDulacLocal,MarVilDulacGeneral,MarVilDulacCoef} it follows that the Dulac map can be written as 
    \[D_i(s;\mu)=\left\{\begin{array}{ll}
        \displaystyle s^{\lambda_i}\bigl(\Delta_{00}^i+\Delta_{10}^is+\mathcal{F}^\infty_{\ell_1}(\mu_0)\bigr) & \text{if } \lambda_i^0>1, \vspace{0.2cm} \\
        \displaystyle s^{\lambda_i}\bigl(\Delta_{00}^i+\Delta_{01}^is^{\lambda_i}+\mathcal{F}^\infty_{\ell_2}(\mu_0)\bigr) & \text{if } \lambda_i^0<1,
    \end{array}\right.\]
for any given $\ell_1\in(1,\min\{\lambda_i^0,2\})$ and $\ell_2\in(\lambda_i^0,\min\{2\lambda_i^0,1\})$. The functions $\Delta_{10}^i$ and $\Delta_{01}^i$ may have some poles and thus may not be well-defined everywhere. Nevertheless, the reader shall see at Section~\ref{Sec:Preliminary} that such a poles will not be a problem in this paper.

In~\cite{MarVilDulacCoef} the authors provided explicit formulas for $\Delta_{10}$ and $\Delta_{01}$. Such formulas depend on some other functions $S_1^i$ and $S_2^i$ satisfying the following relationships:
    \[\Delta_{10}^i=\lambda_i\Delta_{00}^iS_1^i, \quad \Delta_{01}^i=-(\Delta_{00}^i)^2S_2^i.\]
More details are postponed to Section~\ref{Sec:Preliminary}. 

In our second main result we use the refined expression of the Dulac maps to obtain a refined expression for the first return map, which in turn allow us to obtain conditions for higher cyclicities.

\begin{main}\label{Teo:Return1-+Cycl23}
    Let $\{X_\mu\}_{\mu\in\Lambda}$ be a smooth family of planar analytic vector fields having a persistent polycycle $\Gamma$ with hyperbolic saddles $p_1,\dots,p_m,p_{m+1},\dots,p_n$. Let $\mu_0\in\Lambda$ be such that $\lambda_i(\mu_0)<1$ for $i\in\{1,\dots,m\}$ and $\lambda_i(\mu_0)>1$ for $i\in\{m+1,\dots,n\}$. Then the first return map of $\Gamma$ is given by
    \begin{equation}\label{eq:Return1-+}.
        \mathscr{R}(s;\mu)=s^{r(\mu)}\bigl(A_{1,n}+\mathcal{A}s^{\Lambda_{0,m}}+\mathcal{F}^\infty_\ell(\mu_0)\bigr),
    \end{equation}
    for any given $\ell\in(\Lambda_{0,m}^0,\min\{r(\mu_0),2\Lambda_{0,m}^0,1\})$,
    where 
    \begin{equation}\label{22}
        \mathcal{A}=\Lambda_{m,n}A_{1,m}A_{1,n}(S_1^{m+1}-S_2^m).
    \end{equation}
    Moreover, the following holds:
\begin{itemize}
    \item [(a)] ${\rm Cycl}(\Gamma,\mu_0)\leqslant 2$, if $\mathcal{A}(\mu_0)\neq 0$;
    \item [(b)] ${\rm Cycl}(\Gamma,\mu_0)\geqslant 3$ if $r(\mu_0)=A_{1,n}(\mu_0)=1$, $\mathcal{A}(\mu_0)=0$, $r-1,\;A_{1,n}-1,\;\mathcal{A}$ are independent at $\mu_0$ and $\mathscr{R}(\cdot;\mu_0)\not\equiv Id$.
\end{itemize}
\end{main}

Under the hypothesis of Theorem~\ref{Teo:Return1-+Cycl23}, we observe from~\eqref{22} that to calculate the non-leading term of the first return map, it is only necessary to know the non-leading terms of the Dulac maps of indexes $m$ and $m+1$.

The paper is organized as follows. In Section \ref{Sec:Preliminary} we present the fundamental concepts that will be required to the development of the paper, namely: The finitely flat functions and their properties, and the Dulac map of a hyperbolic saddle. In Section \ref{Sec:Returnmap} we prove the technical results on the composition and inverse of Dulac maps that allowed us to obtain the coefficients in the asymptotic expansion of the return map $\mathscr{R}(s;\mu)$. In Section~\ref{Sec:displacement1} we recall the notion of displacement map, also used in the literature to study the cyclicity of persistent polycycles, and we study its coefficients. The proofs of our main results are presented in Section \ref{Sec:Proofs}. In Section~\ref{Sec:displacement2} we state and prove a similar version or our main results for the displacement map and observe that its coefficients are equivalent with the coefficients of the first return map. We conclude the paper in Section \ref{Sec:Application}, presenting an application of our results in the context of Game Theory.

\section{Preliminary results}\label{Sec:Preliminary}

\subsection{Finitely flat functions}\label{Sec:Flat}

We introduce the notion of finitely flat functions that play a substantial role when dealing with the return map of a polycycle.

\begin{defi}
    Consider $K\in\mathbb{Z}_{\geqslant 0}\cup\{\infty\}$ and an open set $U\subset\mathbb{R}^{N}$. We say that a function $\psi(s;\mu)$ belongs to class $\mathscr C^K_{s>0}(U)$ if there exists an open neighborhood $\Omega$ of $\{0\}\times U$ in $\mathbb{R}^{N+1}$ such that $(s;\mu)\mapsto \psi(s,\mu)$ is $\mathscr C^K$ on $\Omega\cap\left\{(0,+\infty)\times U\right\}$.
\end{defi}

\begin{defi}[Finitely flat functions]\label{defiFlat}
	Consider $K\in\mathbb{Z}_{\geqslant 0}\cup\{\infty\}$ and an open set $U\subset\mathbb{R}^{N}$. Given $L\in\mathbb{R}$ and $\mu_0\in U$, we say that $\psi(s;\mu)\in \mathscr C^K_{s>0}(U)$ is $(L,K)$-flat with respect to $s$ at $\mu_0$, and we write $\psi\in\mathcal{F}^{K}_{L}(\mu_0)$, if for each $\nu=(\nu_0,\dots,\nu_{N})\in\mathbb{Z}^{N+1}_{\geqslant 0}$ with $|\nu|\leqslant K$, there exist a neighborhood $V$ of $\mu_0$ and $C$, $s_0>0$ such that
    \[\vert\partial_\nu\psi(s;\mu)\vert:=\left\vert\dfrac{\partial^{|\nu|}\psi}{\partial s^{\nu_0}\partial\mu_1^{\nu_1}\cdots\partial\mu_N^{\nu_N}}(s;\mu)\right\vert\leqslant Cs^{L-\nu_0}\;\text{for all } s\in(0,s_0)\text{ and } \mu\in V.\]
	If $W$ is a (not necessarily open) subset of $U$, then $\mathcal{F}^{\infty}_{L}(W)=\bigcap\limits_{\mu_0\in W}\mathcal{F}^{\infty}_{L}(\mu_0)$. 
\end{defi}

The usefulness of the finitely flat functions is presented in the next result.

\begin{lem}[{\cite[Lemma A.3]{MarVilDulacLocal}}]\label{LemaPropFlat}
	Let $U$ and $U'$ be open sets of $\mathbb{R}^N$ and $\mathbb{R}^{N'}$ respectively and consider $W\subset U$ and $W'\subset U'$. Then, the following holds:
	\begin{itemize}
		\item[(a)] $\mathcal{F}^{K}_{L}(W)\subset\mathcal{F}^{K}_{L}(\hat{W})$ for any $\hat{W}\subset W$;
		\item[(b)]$\mathcal{F}^{K}_{L}(W)\subset\mathcal{F}^{K}_{L}(W\times W')$;
		\item[(c)] $\mathscr C^K(U)\subset\mathcal{F}^{K}_{0}(W)$;
		\item[(d)] If $K\geqslant K'$ and $L\geqslant L'$ then $\mathcal{F}^{K}_{L}(W)\subset \mathcal{F}^{K'}_{L'}(W)$;
		\item[(e)] $\mathcal{F}^{K}_{L}(W)$ is closed under addition;		\item[(f)] If $f\in\mathcal{F}^{K}_{L}(W)$ and $\nu\in\mathbb{Z}^{N+1}_{\geqslant 0}$ with $|\nu|\leqslant K$ then $\partial_{\nu}f\in\mathcal{F}^{K-|\nu|}_{L-\nu_0}(W)$;
		\item[(g)] $\mathcal{F}^{K}_{L}(W)\cdot\mathcal{F}^{K'}_{L'}(W)\subset \mathcal{F}^{K}_{L+L'}(W)$;
		\item [(h)] Assume that $\phi:U'\to U$ is a $\mathscr C^K$ function with $\phi(W')\subset W$ and let us take $g\in \mathcal{F}^{K}_{L'}(W')$ with $L'>0$ and verifying $g(s;\eta)>0$ for all $\eta\in W'$ and $s>0$ small enough. Consider also any $f\in\mathcal{F}^{K}_{L}(W)$. Then $h(s;\eta):=f(g(s;\eta);\phi(\eta))$ is a well-defined function that belongs to $\mathcal{F}^{K}_{LL'}(W')$.
	\end{itemize}
\end{lem}

In what follows we prove another technical result about $\mathcal{F}^K_L(W)$. 

\begin{lem}\label{Lemma1}
    Given $K\in\mathbb{Z}_{\geqslant0}\cup\{\infty\}$, consider $a$, $b$, $\eta$, $\lambda\in\mathscr{C}^K(U)$ such that $b(\mu)\neq0$, $\lambda(\mu)>0$ for every $\mu\in U$  and denote $\lambda^0=\lambda(\mu_0)$. If $L\in(\lambda^0,2\lambda^0)$ then 
        \[\bigl(b+as^\lambda+\mathcal{F}^K_L(\mu_0)\bigr)^\eta=b^\eta+\eta b^{\eta-1}as^\lambda+\mathcal{F}^K_L(\mu_0),\]
    for $s>0$ small enough such that $\bigl|\frac{a}{b}s^\lambda\bigr|<1$.
\end{lem}

\begin{proof}
    We first prove for the case $b(\mu)\equiv 1$. From the Generalized Binomial Theorem~\ref{GBT} (GBT) we have that
        \[(1+as^\lambda)^{-1}=1-as^\lambda+\mathcal{F}^K_L(\mu_0),\]
    for $s>0$ small enough such that $|a s^\lambda|<1$. In particular we have $(1+as^\lambda)^{-1}\in\mathcal{F}^K_0(\mu_0)$. Furthermore it also follows from the GBT that
    \begin{equation}\label{1}
        (1+as^\lambda)^\eta=1+\eta a s^{\lambda}+\mathcal{F}^K_L(\mu_0),
    \end{equation}
    for $s>0$ small enough such that $|as^\lambda|<1$. Hence $(1+as^\lambda)^\eta\in\mathcal{F}^K_0(\mu_0)$. Now observe that
        \[\begin{array}{ll}
            \bigl(1+as^\lambda+\mathcal{F}^K_L(\mu_0)\bigr)^\eta-(1+as^\lambda)^\eta &=(1+as^\lambda)^\eta\bigl[\bigl(1+(1+as^\lambda)^{-1}\mathcal{F}^K_L(\mu_0)\bigr)^\eta-1\bigr] \vspace{0.2cm} \\
            &= (1+as^\lambda)^\eta\bigl[\bigl(1+\mathcal{F}^K_L(\mu_0)\bigr)^\eta-1\bigr] \vspace{0.2cm} \\
            &= (1+as^\lambda)^\eta\mathcal{F}^K_L(\mu_0)=\mathcal{F}^K_L(\mu_0),
        \end{array}\]
    where the second equality follows from $(1+as^\lambda)^{-1}\in\mathcal{F}^K_0(\mu_0)$ in addition with Lemma~\ref{LemaPropFlat}(g), the third equality follows from the GBT and the fourth one following from $(1+as^\lambda)^\eta\in\mathcal{F}^K_0(\mu_0)$ and Lemma~\ref{LemaPropFlat}(g). Thus we conclude that,
        \[\bigl(1+as^\lambda+\mathcal{F}^K_L(\mu_0)\bigr)^\eta=(1+as^\lambda)^\eta+\mathcal{F}^K_L(\mu_0).\]
    This in addition with \eqref{1} and Lemma~\ref{LemaPropFlat}(e) implies that
        \[\bigl(1+as^\lambda+\mathcal{F}^K_L(\mu_0)\bigr)^\eta=1+\eta as^\lambda+\mathcal{F}^K_L(\mu_0).\]
    The general case now follows from observing that
        \[\begin{array}{l}
             \displaystyle \bigl(b+as^\lambda+\mathcal{F}^K_L(\mu_0)\bigr)^\eta = b^\eta \left(1+\frac{a}{b}s^\lambda+\mathcal{F}^K_L(\mu_0)\right)^\eta \vspace{0.2cm} \\
             \qquad\quad\displaystyle= b^\eta\left(1+\eta \frac{a}{b}s^\lambda+\mathcal{F}^K_L(\mu_0)\right) = b^\eta+\eta b^{\eta-1}a s^\lambda+\mathcal{F}^K_L(\mu_0),
        \end{array}\]
    provided $s>0$ is small enough such that $\bigl|\frac{a}{b}s^\lambda\bigr|<1$.
\end{proof}

\begin{defi}\label{Def8}
	The function defined for $s>0$ and $\alpha\in\mathbb{R}$ by means of
	\begin{equation}\label{ERC}
            \omega(s;\alpha)=\left\{\begin{array}{c}
		\frac{s^{-\alpha}-1}{\alpha}\quad\text{if }\alpha\neq 0,\\
		-\ln s \quad\text{if }\alpha=0,
	\end{array}\right.
        \end{equation}
	is called \emph{\'Ecalle--Roussarie compensator}.
\end{defi}

The properties of the \'Ecalle--Roussarie compensator are studied in detail in~\cite[Appendix~A]{MarVilDulacLocal}. We highlight three of these properties in the next lemma.

\begin{lem}[{\cite[Lemma A.4]{MarVilDulacLocal}}]\label{Lema:omega}
The following holds for the \'Ecalle--Roussarie compensator:
\begin{itemize}
\item $\partial_s\omega(s;\alpha)=-s^{-\alpha-1}$; 
\item $\lim\limits_{s\to 0^+}\dfrac{1}{\omega(s;\alpha)}=\max\{-\alpha,0\}$ uniformly on $\alpha\in\mathbb{R}$ and in particular, 
    \[\lim\limits_{(s,\alpha)\to (0^+,0)}\dfrac{1}{\omega(s;\alpha)}=0;\]
    \item $\omega(s;\alpha),\frac{1}{\omega(s;\alpha)}\in\mathcal{F}^\infty_{-\delta}(\{\alpha<\delta\})$ for every $\delta>0$.
\end{itemize} 
\end{lem}

\subsection{The Dulac map}\label{Sec:Dulac}

Since we deal with persistent hyperbolic polycycles, we need to work with the Dulac map and Dulac time associated to hyperbolic saddles. We follow closely the construction made in~\cite{MarVilDulacLocal,MarVilDulacGeneral,MarVilDulacCoef} where the specifics are carried out extensively. We encourage the reader to seek these references for a substantial understanding of the Dulac map and time. Here, we only state the results necessary for our investigation.

We consider an open set $\Lambda\subset\mathbb{R}^N$ and the family $\{X_{\mu}\}_{\mu\in\Lambda}$ of vector fields given by:
\begin{equation}\label{eq:X1}
	X_{\mu}:=\dfrac{1}{x^{n_1}y^{n_2}}\left(xP(x,y;\mu)\partial_x+yQ(x,y;\mu)\partial_y\right).
\end{equation}
Here,
\begin{itemize}
	\item $\mathtt{n}:=(n_1,n_2)\in\mathbb{Z}_{\geqslant 0}^2$;
	\item $P,Q\in \mathscr C^{\infty}(V\times \Lambda),$ for some open set $V\subset\mathbb{R}^2$ containing the origin;
	\item $P(x,0;\mu)>0$ and $Q(0,y;\mu)<0$, for all $(x,0), (0,y)\in V$ and $\mu\in\Lambda$. This means that the origin is a hyperbolic saddle of $x^{n_1}y^{n_2}X_\mu$ with the $y$-axis being the stable manifold and $x$-axis the unstable manifold;
	\item $\lambda(\mu)=-\dfrac{Q(0,0;\mu)}{P(0,0;\mu)}$ is the hyperbolic ratio of the saddle.
\end{itemize}

For $i\in\{1,2\}$, let $\sigma_i:(-\varepsilon,\varepsilon)\times\Lambda\to\Sigma_i$ be transverse sections of $X_{\mu}$ to the axis such that

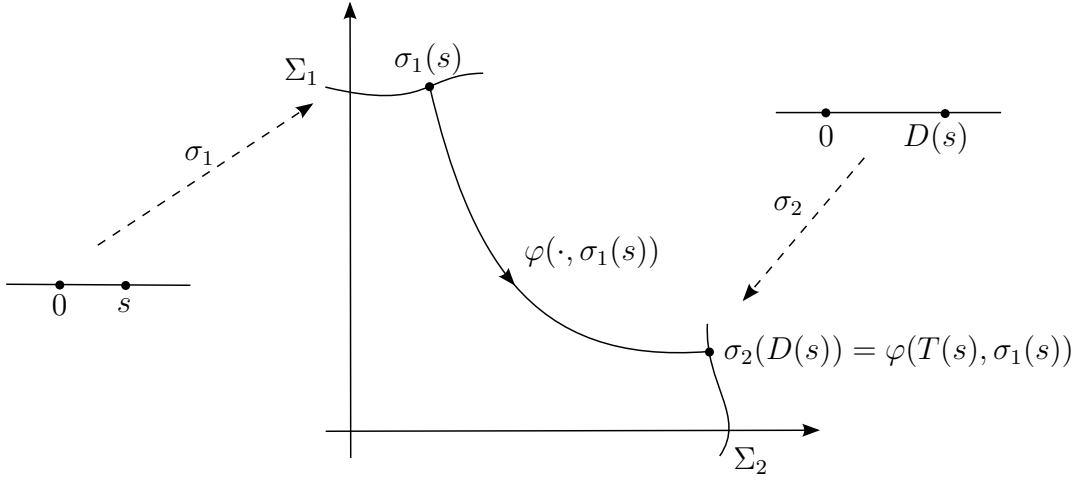
\begin{figure}[ht]
	\begin{center}		
		\begin{overpic}[width=0.8\textwidth]{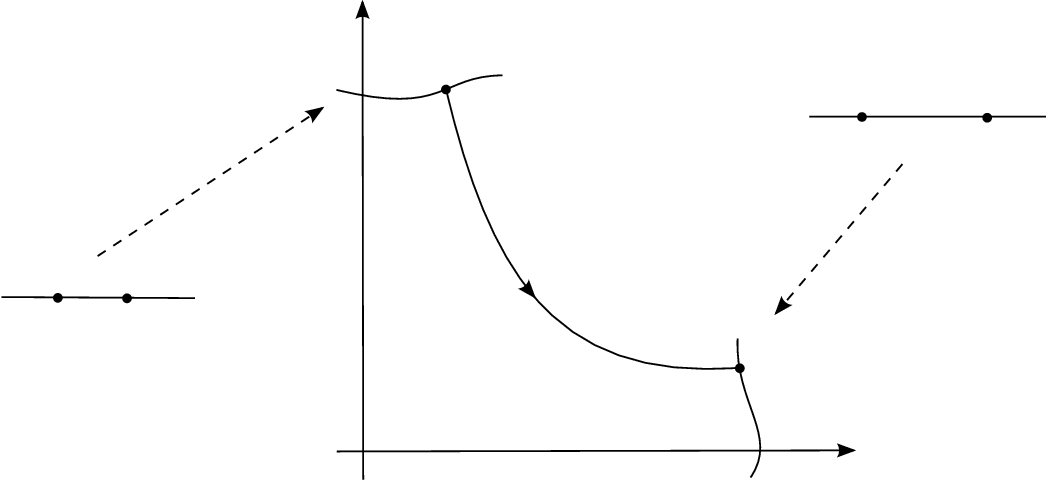}
			\put(4.7,14.5) {$0$}
			\put(11.2,14.8) {$s$}
			\put(18,30) {$\sigma_1$}
			\put(28,38) {$\Sigma_1$}
			\put(52,20) {$\varphi(\cdot,\sigma_1(s))$}
			\put(39,39.3) {$\sigma_1(s)$}
			\put(72,10) {$\sigma_2(D(s))=\varphi(T(s),\sigma_1(s))$}
                \put(73,-1) {$\Sigma_2$}
			\put(77,25) {$\sigma_2$}
			\put(81.5,31.4) {$0$}
			\put(90,31.4) {$D(s)$}
		\end{overpic}
	\end{center}
	\caption{The Dulac map and time.}\label{FigDulac}
\end{figure}

$\sigma_1(0;\mu)\in\{(0,y):y>0\}$ and $\sigma_2(0;\mu)\in\{(x,0):x>0\}$ for all $\mu\in\Lambda$. The Dulac map $D(\cdot;\mu)$ and the Dulac time $T(\cdot;\mu)$ are defined by the following relationship:
    \[\varphi(T(s;\mu),\sigma_1(s;\mu);\mu)=\sigma_2(D(s;\mu);\mu),\;\forall s\in (0,\varepsilon),\]
where $\varphi(t,p_0;\mu)$ is the solution of $X_{\mu}$ with initial condition $\varphi(0,p_0;\mu)=p_0$ (see Figure~\ref{FigDulac}).

The following result is a particular case of Theorem B in \cite{MarVilDulacGeneral}. See also Theorem~$C.5$ and Remark~$1.1$ of \cite{MarVilDulacCoef}.

\begin{teo}\label{Teo:Dulacmap}
Let $D(s;\mu)$ be the Dulac map of the hyperbolic saddle \eqref{eq:X1} from $\Sigma_1$ to $\Sigma_2$. Then, for $\lambda^0=\lambda(\mu_0)$, the following holds.
\begin{itemize}
    \item[$(a)$] For $\lambda^0<1$, and $\ell\in(\lambda^0,\min\{2\lambda^0,1\})$,
        \[D(s;\mu)=s^\lambda\bigl(\Delta_{00}(\lambda,\mu)+\Delta_{01}(\lambda,\mu)s^\lambda+\mathcal{F}_{\ell}^{\infty}(\mu_0)\bigr),\]
where $\Delta_{00}\in \mathscr C^\infty(\{(0,\infty)\}\times\Lambda)$ and $\Delta_{01}\in \mathscr C^\infty(\{(0,\infty)\setminus\mathbb{N}\}\times\Lambda)$. Moreover, $\Delta_{00}$ is strictly positive;
    \item[$(b)$] For $\lambda^0>1$, and $\ell\in(1,\min\{\lambda^0,2\})$,
        \[D(s;\mu)=s^\lambda\bigl(\Delta_{00}(\lambda,\mu)+\Delta_{10}(\lambda,\mu)s+\mathcal{F}_{\ell}^{\infty}(\mu_0)\bigr),\]
where $\Delta_{10}\in \mathscr C^\infty(\{(0,\infty)\setminus\frac{1}{\mathbb{N}}\}\times\Lambda)$;
\item[$(c)$] For $\lambda^0=1$, and $\ell\in (1,2)$,
    \[D(s;\mu)=s^\lambda\bigl(\Delta_{00}(\lambda,\mu)+\mathbf{\Delta}_{10}(\lambda,\mu)s+\mathcal{F}_{\ell}^{\infty}(\mu_0)\bigr),\]
where $\mathbf{\Delta}_{10}=\Delta_{10}+\Delta_{01}(1+\alpha\omega(s;\alpha))$ and $\alpha=1-\lambda$.     
\end{itemize}
\end{teo}

\begin{obs}
    Under the hypothesis of Theorem~\ref{Teo:Dulacmap}$(a)$, although $\Delta_{01}$ may not be well defined for $\lambda\in\mathbb{N}$, these values are unreachable due to the hypothesis of $\lambda^0<1$. More precisely, from the initial condition $\lambda^0<1$ we have that there is a neighborhood of $U$ of $\mu_0$ such that $\lambda<1$ for every $\mu\in U$. Hence, for our purposes in this paper, we can suppose that $\Delta_{01}$ is always well-defined. Similarly for Theorem~\ref{Teo:Dulacmap}$(b)$.
\end{obs}


We observe that that Theorem~\ref{Teo:Dulacmap}, in the way it was stated, applies to hyperbolic saddles at the origin and for which the separatrices are contained in the orthogonal axis. However, this is not a restrictive assumption since we can translate the saddle and rectify its separatrices via a smooth family diffeomorphism, see~\cite[Lemma~$4.3$]{MarVilDulacGeneral}.

\section{The return map of a persistent polycycle}\label{Sec:Returnmap}

Consider a smooth family $\{X_\mu\}_{\mu\in\Lambda}$ of planar smooth vector fields having a persistent polycycle $\Gamma$ with $n$ hyperbolic saddles, namely $p_1,\dots,p_n$. For $i\in\{1,\dots,n\}$, let $\Sigma_i$ be a transversal section to the connection $\gamma_i$ from $p_{i-1}$ to $p_i$ (set $p_0=p_n$), and $D_i=D_i(\cdot;\mu)$ be the corresponding Dulac map, see Figure~\ref{Fig4}. 
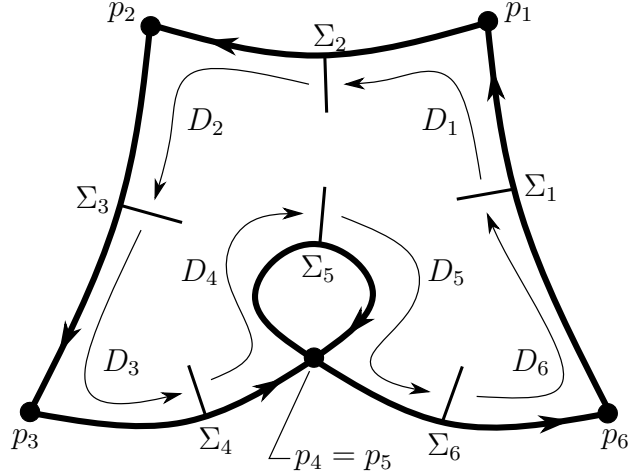
\begin{figure}[ht]
	\begin{center}		
		\begin{overpic}[width=8cm]{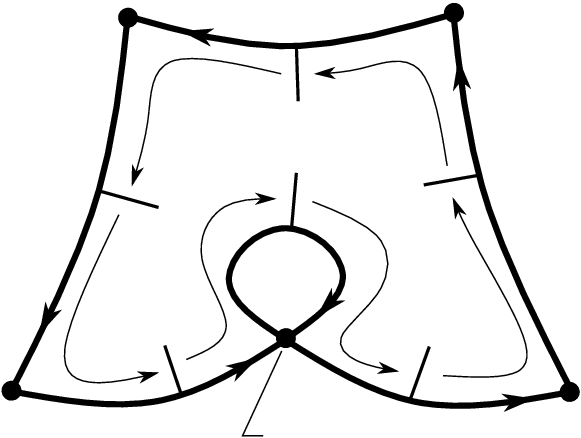} 
                \put(81,74){$p_1$}
                \put(15,74){$p_2$}
                \put(-1,3){$p_3$}
                \put(46,-0.5){$p_4=p_5$}
                \put(97,3){$p_6$}
                \put(84,44){$\Sigma_1$}
                \put(49,69){$\Sigma_2$}
                \put(9.5,42){$\Sigma_3$}
                \put(30,2){$\Sigma_4$}
                \put(47,30.5){$\Sigma_5$}
                \put(68,1){$\Sigma_6$}
                \put(67,55){$D_1$}
                \put(28,55){$D_2$}
                \put(14,15){$D_3$}
                \put(27,30){$D_4$}
                \put(68,30){$D_5$}
                \put(82,15){$D_6$}
		\end{overpic}
	\end{center}
	\caption{Illustration of the Dulac maps of a polycycle $\Gamma$.}\label{Fig4}
\end{figure}

For the remainder of this paper, we denote with a superscript the index of which Dulac map with which we are working, i.e. $\Delta_{jk}^i$ and $S_j^i$ denotes the coefficient $\Delta_{jk}$ and quantity $S_j$ in the Dulac map $D_i(s,\mu)$.

For a point in position $s$ at $\Sigma_1$, we define the return map of $\Gamma$ by
\begin{equation*}
	\mathscr{R}(s;\mu)=D_n\circ\cdots\circ D_1(s;\mu).
\end{equation*}
Thus, it is essential to understand the composition of Dulac maps to investigate the cyclicity as isolated fixed points of $\mathscr{R}$ correspond to limit cycles. 


\subsection{Composition of Dulac maps}

In this section we study the composition of Dulac maps. For our first result, we observe that from Theorem~\ref{Teo:Dulacmap} we have that the leading term of a Dulac map does not depend on the sign of $\lambda_i^0-1$ (it could even be zero). 

\begin{lem}\label{Lemma2}
    Let $\{X_\mu\}_{\mu\in\Lambda}$ be a smooth family of planar smooth vector fields having a persistent polycycle $\Gamma$ with hyperbolic saddles $p_1$ and $p_2$ and consider $\mu_0\in\Lambda$. Then for any given $\ell\in(0,\min\{1,\lambda_1^0,\lambda_1^0\lambda_2^0\})$ it holds
        \[D_2\circ D_1(s;\mu)=s^{\lambda_1\lambda_2}\bigl(\Upsilon_0+\mathcal{F}^\infty_\ell(\mu_0)\bigr),\]
    where $\Upsilon_0=(\Delta_{00}^1)^{\lambda_2}\Delta_{00}^2$.
\end{lem}

\begin{proof}
    Let $\ell\in(0,\min\{1,\lambda_1^0,\lambda_1^0\lambda_2^0\})$. From Theorem~\ref{Teo:Dulacmap} we have (even if $\lambda_1^0=1$ or $\lambda_2^0=1$) that
       \[D_1(s;\mu)=s^{\lambda_1}\bigl(\Delta_{00}^1+\mathcal{F}^\infty_{\ell_1}(\mu_0)\bigr), \quad D_2(s;\mu)=s^{\lambda_2}\bigl(\Delta_{00}^2+\mathcal{F}^\infty_{\ell_2}(\mu_0)\bigr),\]
    for any given $\ell_1\in(0,\min\{1,\lambda_1^0\})$ and $\ell_2\in(0,\min\{1,\lambda_2^0\})$. Observe that
    \[\begin{array}{l}
        D_2\circ D_1(s;\mu) = D_2\bigl(s^{\lambda_1}\bigl(\Delta_{00}^1+\mathcal{F}^\infty_{\ell_1}(\mu_0)\bigr)\bigr) \vspace{0.2cm} \\
               
        \quad=s^{\lambda_1\lambda_2}\bigl(\Delta_{00}^1+\mathcal{F}^\infty_{\ell_1}(\mu_0)\bigr)^{\lambda_2}\bigl(\Delta_{00}^2+\mathcal{F}^\infty_{\ell_3}(\mu_0)\bigr) \vspace{0.2cm} \\
               
        \qquad= s^{\lambda_1\lambda_2}\bigl((\Delta_{00}^1)^{\lambda_2}+\mathcal{F}^\infty_{\ell_1}(\mu_0)\bigr)\bigl(\Delta_{00}^2+\mathcal{F}^\infty_{\ell_3}(\mu_0)\bigr) \vspace{0.2cm} \\
               
        \qquad\quad= s^{\lambda_1\lambda_2}\bigl((\Delta_{00}^1)^{\lambda_2}\Delta_{00}^2+\Delta_{00}^2\mathcal{F}^\infty_{\ell_1}(\mu_0)+(\Delta_{00}^1)^{\lambda_2}\mathcal{F}^\infty_{\ell_3}(\mu_0)+\mathcal{F}^\infty_{\ell_1}(\mu_0)\mathcal{F}^\infty_{\ell_3}(\mu_0)\bigr) \vspace{0.2cm} \\\

        \qquad\qquad= s^{\lambda_1\lambda_2}\bigl((\Delta_{00}^1)^{\lambda_2}\Delta_{00}^2+\mathcal{F}^\infty_{\ell_1}(\mu_0)+\underbrace{\mathcal{F}^\infty_{\ell_3}(\mu_0)+\mathcal{F}^\infty_{\ell_1+\ell_3}(\mu_0)}_{\mathcal{F}^\infty_{\ell_3}(\mu_0)}\bigr) \vspace{0.2cm} \\

        \qquad\qquad\quad= s^{\lambda_1\lambda_2}\bigl((\Delta_{00}^1)^{\lambda_2}\Delta_{00}^2+\mathcal{F}^\infty_{\ell_1}(\mu_0)+\mathcal{F}^\infty_{\ell_3}(\mu_0)\bigr),
    \end{array}\]
    with $\ell_3=\lambda_1^0\ell_2$ following from Lemma~\ref{LemaPropFlat}(h), the third equality following from Lemma~\ref{Lemma1} (with $a=0$), the fifth equality following from Lemma~\ref{LemaPropFlat}(g) and the last equality following from Lemma~\ref{LemaPropFlat}(d). It now follows from Lemma~\ref{LemaPropFlat}(d,e) that
       \[D_2\circ D_1(s,\mu)=s^{\lambda_1\lambda_2}\bigl((\Delta_{00}^1)^{\lambda_2}\Delta_{00}^2+\mathcal{F}^\infty_{\ell_4}(\mu_0)\bigr),\]
    for any given $\ell_4\in(0,\min\{\ell_1,\ell_3\})$. Since we can choose from the beginning any $\ell_2\in(0,\min\{1,\lambda_2^0\})$, it follows that we can take any $\ell_3\in(0,\min\{\lambda_1^0,\lambda_1^0\lambda_2^0\})$. This in addition with the fact that we can choose $\ell_1\in(0,\min\{1,\lambda_1^0\})$ freely implies that we can also choose $\ell_4\in(0,\min\{1,\lambda_1^0,\lambda_1^0\lambda_2^0\})$ freely. In particular, we can take $\ell_4=\ell$.
\end{proof}

In the next result we apply induction on Lemma~\ref{Lemma2} to obtain a general formula for the composition of $n$ Dulac maps. To this end, we recall that
    \[A_{j,k}=\prod_{i=j}^{k}(\Delta_{00}^{i})^{\Lambda_{i,k}}, \quad \Lambda_{i,k}=\prod_{j=i+1}^{k}\lambda_j, \quad \Lambda_{kk}=1, \quad \Lambda_{i,k}^0=\Lambda_{i,k}(\mu_0).\]
\begin{corol}\label{Coro0}
    Let $\{X_\mu\}_{\mu\in\Lambda}$ be a smooth family of planar smooth vector fields having a persistent polycycle $\Gamma$ with hyperbolic saddles $p_1,\dots,p_n$ and consider $\mu_0\in\Lambda$. Then for any given $\ell\in(0,\min\{\Lambda_{0,i}^0\colon i\in\{0,\dots,n\}\})$ it holds
    \begin{equation}\label{21}
        D_n\circ\ldots\circ D_1(s;\mu)=s^{\Lambda_{0,n}}\bigl(A_{1,n}+\mathcal{F}^\infty_\ell(\mu_0)\bigr).
    \end{equation}
\end{corol}

\begin{proof}
    For simplicity we write,
        \[D_1(s;\mu)=s^{\lambda_1}\bigl(a_1+\mathcal{F}^\infty_{\ell_1}(\mu_0)\bigr), \quad D_2(s;\mu)=s^{\lambda_2}\bigl(a_2+\mathcal{F}^\infty_{\ell_2}(\mu_0)\bigr).\]
    It follows from Lemma~\ref{Lemma2} that
        \[D_2\circ D_1(s;\mu)=s^{\lambda_1\lambda_2}\bigl(a_1^{\lambda_2}a_2+\mathcal{F}^\infty_{\ell_{1,2}}(\mu_0)\bigr)=s^{\Lambda_{0,2}}(A_{1,2}+\mathcal{F}^\infty_{\ell_{1,2}}(\mu_0)\bigr),\]
    for any given $\ell_{1,2}\in(0,\min\{1,\lambda_1^0,\lambda_2^0\})$. Suppose that
    \begin{equation}\label{20}
        D_{n-1}\circ\dots\circ D_1(s;\mu)=s^{\Lambda_{0,n-1}}\bigl(A_{1,n-1}+\mathcal{F}^\infty_{\ell_{1,n-1}}(\mu_0)\bigr),
    \end{equation}
    and let
        \[D_n(s;\mu)=s^{\lambda_n}(a_n+\mathcal{F}^\infty_{\ell_n}(\mu_0)\bigr),\]
    for any given $\ell_{1,n-1}\in(0,\min\{\Lambda_{0,i}^0\colon i\in\{0,\dots,n-1\}\})$ and $\ell_n\in(0,\min\{1,\lambda_n^0\})$. Since \eqref{20} is also of Dulac-type, from Lemma~\ref{Lemma2} we have that
        \[D_n\circ(D_{n-1}\circ\dots\circ D_1)(s;\mu)=\bigl(A_{1,n-1}^{\lambda_n}a_n+\mathcal{F}^\infty_{\ell_{1,n}}(\mu_0)\bigr)=\bigl(A_{1,n}+\mathcal{F}^\infty_{\ell_{1,n}}(\mu_0)\bigr),\]
    for any given $\ell_{1,n}\in(0,\min\{\Lambda_{0,i}^0\colon i\in\{0,\dots,n\}\})$. The proof now follows by induction.  
\end{proof}

We observe that formulas similar to~\eqref{21} were already obtained in the literature. See~\cite[p.~$726$]{MarVilDulacGeneral} and~\cite[p.~$12$]{QMV}. Nevertheless, as far as we know the explicit interval associated with $\ell$ is a new result.

In the following results we shall include the next term of the Dulac maps in the computation. We recall that from Theorem~\ref{Teo:Dulacmap} it follows that such a term depend on the sign of $\lambda_i^0-1$. Hence, the compositions must be studied in a case-by-case scenario. Moreover, different from the previous results, from now on in this section we shall assume $\lambda_1^0\neq1$ and $\lambda_2^0\neq1$ for simplicity.

\begin{lem}\label{Lemma3}
    Let $\{X_\mu\}_{\mu\in\Lambda}$ be a smooth family of planar smooth vector fields having a persistent polycycle $\Gamma$ with hyperbolic saddles $p_1$ and $p_2$. Let $\mu_0\in\Lambda$ be such that $\lambda_1^0>1$ and $\lambda_2^0>1$. Then for any given $\ell\in(1,\min\{\lambda_1^0,2\})$ it holds
        \[D_2\circ D_1(s;\mu)=s^{\lambda_1\lambda_2}\bigl(\Upsilon_0+\Upsilon_1 s+\mathcal{F}^\infty_\ell(\mu_0)\bigr),\]
    where $\Upsilon_0=(\Delta_{00}^1)^{\lambda_2}\Delta_{00}^2$ and $\Upsilon_1=\lambda_2(\Delta_{00}^1)^{\lambda_2-1}\Delta_{00}^2\Delta_{10}^1$.
\end{lem}

\begin{proof}
    Let $\ell\in(1,\min\{\lambda_1^0,2\})$. Since $\lambda_1^0>1$ and $\lambda_2^0>1$, from Theorem~\ref{Teo:Dulacmap} we have that
        \[D_1(s;\mu)=s^{\lambda_1}\bigl(\Delta_{00}^1+\Delta_{10}^1s+\mathcal{F}^\infty_{\ell_1}(\mu_0)\bigr), \quad D_2(s;\mu)=s^{\lambda_2}\bigl(\Delta_{00}^2+\Delta_{10}^2s+\mathcal{F}^\infty_{\ell_2}(\mu_0)\bigr),\]
    for any given $\ell_1\in(1,\min\{\lambda_1^0,2\})$ and $\ell_2\in(1,\min\{\lambda_2^0,2\})$. In particular, for any given $\ell_1\in(\ell,\min\{\lambda_1^0,2\})$. Observe that
        \begin{equation}\label{2}
            \begin{array}{l}
               D_2\circ D_1(s;\mu) = D_2\bigl(s^{\lambda_1}\bigl(\Delta_{00}^1+\Delta_{10}^1+\mathcal{F}^\infty_{\ell_1}(\mu_0)\bigr)\bigr) \vspace{0.2cm} \\
               
               \quad=s^{\lambda_1\lambda_2}\bigl(\Delta_{00}^1+\Delta_{10}^1s+\mathcal{F}^\infty_{\ell_1}(\mu_0)\bigr)^{\lambda_2}\bigl(\Delta_{00}^2+\Delta_{10}^2s^{\lambda_1}\bigl(\Delta_{00}^1+\Delta_{10}^1s+\mathcal{F}^\infty_{\ell_1}(\mu_0)\bigr)+\mathcal{F}^\infty_{\ell_3}(\mu_0)\bigr),
            \end{array}
        \end{equation}
    with $\ell_3=\lambda_1^0\ell_2$ following from Lemma~\ref{LemaPropFlat}(h). Observe that $\ell_3>\lambda_1^0$. Applying Lemma~\ref{Lemma1} at \eqref{2} we obtain
        \begin{equation}\label{3}                                              
            \begin{array}{l}    
                s^{\lambda_1\lambda_2}\bigl((\Delta_{00}^1)^{\lambda_2}+\lambda_2(\Delta_{00}^1)^{\lambda_2-1}\Delta_{10}^1s+\mathcal{F}^\infty_{\ell_1}(\mu_0)\bigr)\cdot \vspace{0.2cm} \\
                \quad\cdot\bigl(\Delta_{00}^2+\Delta_{10}^2\Delta_{00}^1s^{\lambda_1}+\underbrace{\Delta_{10}^2\Delta_{10}^1s^{\lambda_1+1}+\Delta_{10}^2s^{\lambda_1}\mathcal{F}^\infty_{\ell_1}(\mu_0)+\mathcal{F}^\infty_{\ell_3}(\mu_0)}_{\mathcal{F}^\infty_{\ell_4}(\mu_0)}\bigr) \vspace{0.2cm} \\

                \qquad=s^{\lambda_1\lambda_2}\bigl((\Delta_{00}^1)^{\lambda_2}+\lambda_2(\Delta_{00}^1)^{\lambda_2-1}\Delta_{10}^1s+\mathcal{F}^\infty_{\ell_1}(\mu_0)\bigr)\bigl(\Delta_{00}^2+\Delta_{10}^2\Delta_{00}^1s^{\lambda_1}+\mathcal{F}^\infty_{\ell_4}(\mu_0)\bigr),
            \end{array}
        \end{equation}
    for any given $\ell_4\in(\lambda_1^0,\min\{1+\lambda_1^0,\ell_1+\lambda_1^0,\ell_3\})$, due to Lemma~\ref{LemaPropFlat}(d,g). Expanding the last two factors of \eqref{3} we obtain
            \[s^{\lambda_1\lambda_2}\bigl(\Delta_{00}^2(\Delta_{00}^1)^{\lambda_2}+\lambda_2\Delta_{00}^2(\Delta_{00}^1)^{\lambda_2-1}\Delta_{10}^1s+\mathcal{F}^\infty_{\ell_5}(\mu_0)\bigr),\]
    for any given 
            \[\ell_5\in(1,\min\{\lambda_1^0,\ell_4,1+\lambda_1^0,1+\ell_4,\ell_1,\ell_1+\lambda_1^0,\ell_1+\ell_4\})=(1,\min\{\lambda_1^0,\ell_1,\ell_4\})=(1,\ell_1).\]
    In particular for $\ell_5=\ell$.
\end{proof}

\begin{corol}\label{Coro1}
    Let $\{X_\mu\}_{\mu\in\Lambda}$ be a smooth family of planar smooth vector fields having a persistent polycycle $\Gamma$ with hyperbolic saddles $p_1,\dots,p_n$. Let $\mu_0\in\Lambda$ be such that $\lambda_i^0>1$ for $i\in\{1,\dots,m\}$. Then for any given $\ell\in(1,\min\{\lambda_1^0,2\})$ it holds
        \[D_n\circ\ldots\circ D_1(s;\mu)=s^{\Lambda_{0,n}}\bigl(A_{1,n}+B_{1,n}s+\mathcal{F}^\infty_\ell(\mu_0)\bigr),\]
    where 
        \[B_{j,k}=\Lambda_{j,k}\frac{\Delta_{10}^j}{\Delta_{00}^j}A_{j,k}, \quad A_{j,k}=\prod_{i=j}^{k}(\Delta_{00}^{i})^{\Lambda_{i,k}}, \quad \Lambda_{i,k}=\prod_{j=i+1}^{k}\lambda_j,\;\Lambda_{kk}=1.\]
\end{corol}

\begin{proof}
    It follows from Lemma~\ref{Lemma3} that in this case the composition of Dulac maps is also of Dulac-type. Therefore the proof follows by induction. More precisely if for simplicity we write 
        \[D_1(s;\mu)=s^{\lambda_1}\bigl(a_1+b_1s+\mathcal{F}^\infty_{\ell_1}(\mu_0)\bigr), \quad D_2(s;\mu)=s^{\lambda_2}\bigl(a_2+b_2s+\mathcal{F}^\infty_{\ell_2}(\mu_0)\bigr),\]
    then it follows from Lemma~\ref{Lemma3} that,
        \[D_2\circ D_1(s;\mu)=s^{\lambda_1\lambda_2}\left(a_1^{\lambda_2}a_2+\lambda_2\frac{b_1}{a_1}\bigl(a_1^{\lambda_2}a_2\bigr)s+\mathcal{F}^\infty_{\ell_1}(\mu_0)\right)=s^{\Lambda_{0,2}}\bigl(A_{1,2}+B_{1,2}s+\mathcal{F}^\infty_{\ell_1}(\mu_0)\bigr).\]
    Suppose therefore that
        \[D_{n-1}\circ\ldots\circ D_1(s;\mu)=s^{\Lambda_{0,n-1}}\bigl(A_{1,n-1}+B_{1,n-1}s+\mathcal{F}^\infty_{\ell_1}(\mu_0)\bigr),\]
    and let
        \[D_n(s;\mu)=s^{\lambda_n}\bigl(a_n+b_ns+\mathcal{F}^\infty_{\ell_n}(\mu_0)\bigr).\]
    From Lemma~\ref{Lemma3} we have that
        \[\begin{array}{l}
            \displaystyle D_n\circ(D_{n-1}\circ\ldots\circ D_1)(s;\mu) \vspace{0.2cm} \\
            
            \displaystyle\quad=s^{\Lambda_{0,n-1}\lambda_n}\left(A_{1,n-1}^{\lambda_n}a_n+\lambda_n\frac{B_{1,n-1}}{A_{1,n-1}}\bigl(A_{1,n-1}^{\lambda_n}a_n\bigr)s+\mathcal{F}^\infty_{\ell_1}(\mu_0)\right) \vspace{0.2cm} \\

            \displaystyle \qquad=s^{\Lambda_{0,n}}\left(A_{1,n}+\lambda_n\Lambda_{1,n-1}\frac{b_1}{a_1}A_{1,n}s+\mathcal{F}^\infty_{\ell_1}(\mu_0)\right) \vspace{0.2cm} \\
            
            \displaystyle\qquad\quad=s^{\Lambda_{0,n}}\bigl(A_{1,n}+B_{1,n}s+\mathcal{F}^\infty_{\ell_1}(\mu_0)\bigr),
        \end{array}\]
    proving the result.
\end{proof}

\begin{lem}\label{Lemma4}
    Let $\{X_\mu\}_{\mu\in\Lambda}$ be a smooth family of planar smooth vector fields having a persistent polycycle $\Gamma$ with hyperbolic saddles $p_1$ and $p_2$. Let $\mu_0\in\Lambda$ be such that $\lambda_1^0<1$ and $\lambda_2^0<1$. Then for any given $\ell\in(\lambda_1^0\lambda_2^0,\min\{\lambda_1^0,2\lambda_1^0\lambda_2^0\})$ it holds
        \[D_2\circ D_1(s;\mu)=s^{\lambda_1\lambda_2}\bigl(\Upsilon_0+\Upsilon_2 s^{\lambda_1\lambda_2}+\mathcal{F}^\infty_\ell(\mu_0)\bigr),\]
    where $\Upsilon_0=(\Delta_{00}^1)^{\lambda_2}\Delta_{00}^2$ and $\Upsilon_2=(\Delta_{00}^1)^{2\lambda_2}\Delta_{01}^2$.
\end{lem}

\begin{proof}
    Let $\ell\in(\lambda_1^0\lambda_2^0,\min\{\lambda_1^0,2\lambda_1^0\lambda_2^0\})$. Since $\lambda_1^0<1$ and $\lambda_2^0<1$, from Theorem~\ref{Teo:Dulacmap} we have that
        \[D_1(s;\mu)=s^{\lambda_1}\bigl(\Delta_{00}^1+\Delta_{01}^1s^{\lambda_1}+\mathcal{F}^\infty_{\ell_1}(\mu_0)\bigr), \quad D_2(s;\mu)=s^{\lambda_2}\bigl(\Delta_{00}^2+\Delta_{01}^2s^{\lambda_2}+\mathcal{F}^\infty_{\ell_2}(\mu_0)\big),\]
    for any given $\ell_1\in(\lambda_1^0,\min\{2\lambda_1^0,1\})$ and $\ell_2\in(\lambda_2^0,\min\{2\lambda_2^0,1\})$. Similarly to Lemma~\ref{Lemma4} observe that
        \[\begin{array}{l}
            D_2\circ D_1(s;\mu)=s^{\lambda_1\lambda_2}\bigl(\Delta_{00}^1+\Delta_{01}^1s^{\lambda_1}+\mathcal{F}^\infty_{\ell_1}(\mu_0)\bigr)^{\lambda_2}\cdot \vspace{0.2cm} \\
            
            \quad\cdot\bigl(\Delta_{00}^2+\Delta_{01}^2s^{\lambda_1\lambda_2}\bigl(\Delta_{00}^1+\Delta_{01}^1s^{\lambda_1}+\mathcal{F}^\infty_{\ell_1}(\mu_0)\bigr)^{\lambda_2}+\mathcal{F}^\infty_{\ell_3}(\mu_0)\bigr) \vspace{1cm} \\

            \qquad=s^{\lambda_1\lambda_2}\bigl((\Delta_{00}^1)^{\lambda_2}+\lambda_2(\Delta_{00}^1)^{\lambda_2-1}\Delta_{01}^1s^{\lambda_1}+\mathcal{F}^\infty_{\ell_1}(\mu_0)\bigr)\cdot \vspace{0.2cm} \\
            \qquad\quad\cdot\bigl(\Delta_{00}^2+\Delta_{01}^2s^{\lambda_1\lambda_2}\bigl((\Delta_{00}^1)^{\lambda_2}+\lambda_2(\Delta_{00}^1)^{\lambda_2-1}\Delta_{01}^1s^{\lambda_1}+\mathcal{F}^\infty_{\ell_1}(\mu_0)\bigr)+\mathcal{F}^\infty_{\ell_3}(\mu_0)\bigr) \vspace{1cm} \\

            \qquad\qquad=s^{\lambda_1\lambda_2}\bigl((\Delta_{00}^1)^{\lambda_2}+\lambda_2(\Delta_{00}^1)^{\lambda_2-1}\Delta_{01}^1s^{\lambda_1}+\mathcal{F}^\infty_{\ell_1}(\mu_0)\bigr) \vspace{0.2cm} \\
            
            \qquad\qquad\quad\cdot\bigl(\Delta_{00}^2+(\Delta_{00}^1)^{\lambda_2}\Delta_{01}^2s^{\lambda_1\lambda_2}+\mathcal{F}^\infty_{\ell_4}(\mu_0)\bigr) \vspace{1cm} \\

            \qquad\qquad\qquad=s^{\lambda_1\lambda_2}\bigl(\Delta_{00}^2(\Delta_{00}^1)^{\lambda_2}+(\Delta_{00}^1)^{2\lambda_2}\Delta_{01}^2s^{\lambda_1\lambda_2}+\mathcal{F}^\infty_{\ell_5}(\mu_0)\bigr),
        \end{array}\]
    with $\ell_3=\lambda_1^0\ell_2$ and for any given $\ell_4\in(\lambda_1^0\lambda_2^0,\min\{\lambda_1^0\lambda_2^0+\lambda_1^0,\ell_3\})$ and $\ell_5\in(\lambda_1^0\lambda_2^0,\min\{\lambda_1^0,\ell_4\})$. Observe that the possibility to take any $\ell_2\in(\lambda_2^0,\min\{2\lambda_2^0,1\})$ implies that we can take any $\ell_3\in(\lambda_1^0\lambda_2^0,\min\{2\lambda_1^0\lambda_2^0,\lambda_1^0\})$, which in turn implies that we can take any
        \[\ell_4\in(\lambda_1^0\lambda_2^0,\min\{\lambda_1^0\lambda_2^0+\lambda_1^0,2\lambda_1^0\lambda_2^0,\lambda_1^0\})=(\lambda_1^0\lambda_2^0,\min\{2\lambda_1^0\lambda_2^0,\lambda_1^0\}).\]
    This in turn implies that we can take any $\ell_5\in(\lambda_1^0\lambda_2^0,\min\{\lambda_1^0,2\lambda_1^0\lambda_2^0\})$. In particular, we can take $\ell_5=\ell$.
\end{proof}

\begin{corol}\label{Coro2}
    Let $\{X_\mu\}_{\mu\in\Lambda}$ be a smooth family of planar smooth vector fields having a persistent polycycle $\Gamma$ with hyperbolic saddles $p_1,\dots,p_n$. Let $\mu_0\in\Lambda$ be such that $\lambda_i^0<1$ for $i\in\{1,\dots,n\}$. Then for any given $\ell\in(\Lambda_{0,n}^0,\min\{\Lambda_{0,n-1}^0,2\Lambda_{0,n}^0\})$ it holds
        \[D_n\circ\ldots\circ D_1(s;\mu)=s^{\Lambda_{0,n}}\bigl(A_{1,n}+C_{1,n}s^{\Lambda_{0,n}}+\mathcal{F}^\infty_\ell(\mu_0)\bigr),\]
    where 
        \[C_{j,k}=A_{j,k-1}^{2\lambda_k}\Delta_{01}^k, \quad A_{j,k}=\prod_{i=j}^{k}(\Delta_{00}^{i})^{\Lambda_{i,k}}, \quad \Lambda_{i,k}=\prod_{j=i+1}^{k}\lambda_j,\;\Lambda_{kk}=1.\]
\end{corol}

\begin{proof}
    Similarly to the proof of Corollary~\ref{Coro1}, observe that if for simplicity we write
        \[D_1(s;\mu)=s^{\lambda_1}\bigl(a_1+c_1s^{\lambda_1}+\mathcal{F}^\infty_{\ell_1}(\mu_0)\bigr), \quad D_2(s;\mu)=s^{\lambda_2}\bigl(a_2+c_2s^{\lambda_2}+\mathcal{F}^\infty_{\ell_2}(\mu_0)\bigr),\]
    then from Lemma~\ref{Lemma4} we have
        \[D_2\circ D_1(s;\mu)=s^{\lambda_1\lambda_2}\bigl(a_1^{\lambda_2}a_2+a_1^{2\lambda_2}c_2s^{\lambda_1\lambda_2}+\mathcal{F}^\infty_{\ell_{1,2}}(\mu_0)\bigr)=s^{\Lambda_{0,2}}\bigl(A_{1,2}+C_{1,2}s^{\Lambda_{0,2}}+\mathcal{F}^\infty_{\ell_{1,2}}(\mu_0)\bigr),\]
    for any $\ell_{1,2}\in(\Lambda_{0,2}^0,\min\{\Lambda_{0,1}^0,2\Lambda_{0,2}^0\})$. Suppose that
        \[D_{n-1}\circ\ldots\circ D_1(s;\mu)=s^{\Lambda_{0,n-1}}\bigl(A_{1,n-1}+C_{1,n-1}s^{\Lambda_{0,n-1}}+\mathcal{F}^\infty_{\ell_{n-1}}(\mu_0)\big),\]
    with $\ell_{n-1}\in(\Lambda_{0,n-1}^0,\min\{\Lambda_{0,n-2}^0,2\Lambda_{0,n-1}^0\})$ and let
        \[D_n(s;\mu)=s^{\lambda_n}\bigl(a_n+c_ns^{\lambda_n}+\mathcal{F}^\infty_{\ell_n}(\mu_0)\bigr).\]
    From Lemma~\ref{Lemma4} we have
        \[\begin{array}{ll}
            D_n\circ(D_{n-1}\circ\ldots\circ D_1)(s;\mu) &=s^{\Lambda_{0,n-1}\lambda_n}\bigl(A_{1,n-1}^{\lambda_n}a_n+A_{1,n-1}^{2\lambda_n}c_ns^{\Lambda_{0,n-1}\lambda_n}+\mathcal{F}^\infty_{\ell_{1,n}}(\mu_0)\bigr) \vspace{0.2cm} \\

            &=s^{\Lambda_{0,n}}\bigl(A_{1,n}+C_{1,n}s^{\Lambda_{0,n}}+\mathcal{F}^\infty_{\ell_{1,n}}(\mu_0)\bigr),
        \end{array}\]
    with $\ell_{1,n}\in(\Lambda_{0,n}^0\min\{\Lambda_{0,n-1}^0,2\Lambda_{0,n}^0\})$. The result now follows by induction.
\end{proof}

\begin{lem}\label{Lemma5}
    Let $\{X_\mu\}_{\mu\in\Lambda}$ be a smooth family of planar smooth vector fields having a persistent polycycle $\Gamma$ with hyperbolic saddles $p_1$ and $p_2$. Let $\mu_0\in\Lambda$ be such that $\lambda_1^0>1$ and $\lambda_2^0<1$. Then
        \[D_2\circ D_1(s;\mu)=\left\{\begin{array}{ll}
            s^{\lambda_1\lambda_2}\bigl(\Upsilon_0+\Upsilon_1 s+\mathcal{F}^\infty_\ell(\mu_0)\bigr), & \text{if } \lambda_1^0\lambda_2^0>1, \vspace{0.2cm} \\
            
            s^{\lambda_1\lambda_2}\bigl(\Upsilon_0+\Upsilon_\omega s+\mathcal{F}^\infty_{\ell'}(\mu_0)\bigr), & \text{if } \lambda_1^0\lambda_2^0=1, \vspace{0.2cm} \\
            
            s^{\lambda_1\lambda_2}\bigl(\Upsilon_0+\Upsilon_2 s^{\lambda_1\lambda_2}+\mathcal{F}^\infty_{\ell''}(\mu_0)\bigr), & \text{if } \lambda_1^0\lambda_2^0<1, 
        \end{array}\right.\]
    for any given 
        \[\ell\in(1,\min\{\lambda_1^0\lambda_2^0,2\}), \quad \ell'\in(1,\min\{\lambda_1^0,2\}), \quad \ell''\in(\lambda_1^0\lambda_2^0,\min\{2\lambda_1^0\lambda_2^0,1\});\]
    where 
        \[\Upsilon_0=(\Delta_{00}^1)^{\lambda_2}\Delta_{00}^2, \quad \Upsilon_1=\lambda_2(\Delta_{00}^1)^{\lambda_2-1}\Delta_{00}^2\Delta_{10}^1, \quad \Upsilon_2=(\Delta_{00}^1)^{2\lambda_2}\Delta_{01}^2,\] 
    $\Upsilon_\omega=\Upsilon_1+\bigl(1+\alpha\omega(s;\alpha)\bigr)\Upsilon_2$ and $\alpha=1-\lambda_1\lambda_2$.
\end{lem}

\begin{proof}
    Since $\lambda_1^0>1$ and $\lambda_2^0<1$, from Theorem~\ref{Teo:Dulacmap} we have that
          \[D_1(s;\mu)=s^{\lambda_1}\bigl(\Delta_{00}^1+\Delta_{10}^1s+\mathcal{F}^\infty_{\ell_1}(\mu_0)\bigr), \quad D_2(s;\mu)=s^{\lambda_2}\bigl(\Delta_{00}^2+\Delta_{01}^2s^{\lambda_2}+\mathcal{F}^\infty_{\ell_2}(\mu_0)\bigr),\]
    for any given $\ell_1\in(1,\min\{\lambda_1^0,2\})$ and $\ell_2\in(\lambda_2^0,\min\{2\lambda_2^0,1\})$. Similarly Lemmas~\ref{Lemma3} and \ref{Lemma4} observe that
        \[\begin{array}{l}
            D_2\circ D_1(s;\mu)=s^{\lambda_1\lambda_2}\bigl(\Delta_{00}^1+\Delta_{10}^1s+\mathcal{F}^\infty_{\ell_1}(\mu_0)\bigr)^{\lambda_2}\cdot \vspace{0.2cm} \\
            \quad\cdot\bigl(\Delta_{00}^2+\Delta_{01}^2s^{\lambda_1\lambda_2}\bigl(\Delta_{00}^1+\Delta_{10}^1s+\mathcal{F}^\infty_{\ell_1}(\mu_0)\bigr)^{\lambda_2}+\mathcal{F}^\infty_{\ell_3}(\mu_0)\bigr) \vspace{1cm} \\

            \qquad =s^{\lambda_1\lambda_2}\bigl((\Delta_{00}^1)^{\lambda_2}+\lambda_2(\Delta_{00}^1)^{\lambda_2-1}\Delta_{10}^1s+\mathcal{F}^\infty_{\ell_1}(\mu_0)\bigr)\cdot \vspace{0.2cm} \\
            \qquad\quad\cdot\bigl(\Delta_{00}^2+\Delta_{01}^2s^{\lambda_1\lambda_2}\bigl((\Delta_{00}^1)^{\lambda_2}+\lambda_2(\Delta_{00}^1)^{\lambda_2-1}\Delta_{10}^1s+\mathcal{F}^\infty_{\ell_1}(\mu_0)\bigr)+\mathcal{F}^\infty_{\ell_3}(\mu_0)\bigr) \vspace{1cm} \\

            \qquad\qquad=s^{\lambda_1\lambda_2}\bigl((\Delta_{00}^1)^{\lambda_2}+\lambda_2(\Delta_{00}^1)^{\lambda_2-1}\Delta_{10}^1s+\mathcal{F}^\infty_{\ell_1}(\mu_0)\bigr)\bigl(\Delta_{00}^2+(\Delta_{00}^1)^{\lambda_2}\Delta_{01}^2s^{\lambda_1\lambda_2}+\mathcal{F}^\infty_{\ell_4}(\mu_0)\bigr).
        \end{array}\]
    with $\ell_3\in(\lambda_1^0\lambda_2^0,\min\{2\lambda_1^0\lambda_2^0,\lambda_1^0\})$ and
        \[\ell_4\in(\lambda_1^0\lambda_2^0,\min\{\lambda_1^0\lambda_2^0+1,\ell_3\})=(\lambda_1^0\lambda_2^0,\min\{\lambda_1^0\lambda_2^0+1,2\lambda_1^0\lambda_2^0,\lambda_1^0\}).\]
    So far we have proved that $D_2\circ D_1(s;\mu)$ can be expressed as,
        \begin{equation}\label{5}
            s^{\lambda_1\lambda_2}\bigl((\Delta_{00}^1)^{\lambda_2}+\lambda_2(\Delta_{00}^1)^{\lambda_2-1}\Delta_{10}^1s+\mathcal{F}^\infty_{\ell_1}(\mu_0)\bigr)\bigl(\Delta_{00}^2+(\Delta_{00}^1)^{\lambda_2}\Delta_{01}^2s^{\lambda_1\lambda_2}+\mathcal{F}^\infty_{\ell_4}(\mu_0)\bigr).
        \end{equation}
    However we cannot expand these last two factors in a unique way because the next term after the leading one depend on the sign of $1-\lambda_1^0\lambda_2^0$. Hence we need to continue in a case-by-case basis. 
    
    If $\lambda_1^0\lambda_2^0>1$ then we can expand \eqref{5} in to
        \begin{equation}\label{6}
            D_2\circ D_1(s;\mu)=s^{\lambda_1\lambda_2}\bigl(\Delta_{00}^2(\Delta_{00}^1)^{\lambda_2}+\lambda_2\Delta_{00}^2(\Delta_{00}^1)^{\lambda_2-1}\Delta_{10}^1s+\mathcal{F}^\infty_{\ell_5}(\mu_0)\bigr),
        \end{equation}
    for any given $\ell_5\in(1,\min\{\lambda_1^0\lambda_2^0,\ell_1,\ell_4\})=(1,\min\{\lambda_1\lambda_2,2\})$. 
    
    If $\lambda_1^0\lambda_2^0<1$ then we can expand \eqref{5} in to
        \begin{equation}\label{7}
            D_2\circ D_1(s;\mu)=s^{\lambda_1\lambda_2}\bigl(\Delta_{00}^2(\Delta_{00}^1)^{\lambda_2}+(\Delta_{00}^1)^{2\lambda_2}\Delta_{01}^2s^{\lambda_1\lambda_2}+\mathcal{F}^\infty_{\ell_6}(\mu_0)\bigr),
        \end{equation}
    for any given $\ell_6\in(\lambda_1^0\lambda_2^0,\min\{\ell_1,\ell_4,1\})=(\lambda_1^0\lambda_2^0,\min\{2\lambda_1^0\lambda_2^0,1\})$. 
    
    If $\lambda_1^0\lambda_2^0=1$ then let $\alpha=1-\lambda_1\lambda_2$ and observe that
        \begin{equation}\label{8}
            s^{-\alpha}=1+\alpha\omega(s;\alpha),
        \end{equation}
    where we recall that $\omega(s;\alpha)$ is the \'Ecalle--Roussarie compensator~\eqref{ERC}. Since $\lambda_1^0\lambda_2^0=1$ it follows that we cannot isolate the monomials of $s^1$ and $s^{\lambda_1\lambda_2}$ from each other, as in \eqref{6} and \eqref{7}. Hence we expand \eqref{5} in to 
        \begin{equation}\label{9}
        \begin{array}{l}
            D_2\circ D_1(s;\mu) \vspace{0.2cm} \\
            \quad=s^{\lambda_1\lambda_2}\bigl(\underbrace{\Delta_{00}^2(\Delta_{00}^1)^{\lambda_2}}_{\Upsilon_0}+\underbrace{\lambda_2\Delta_{00}^2(\Delta_{00}^1)^{\lambda_2-1}\Delta_{10}^1}_{\Upsilon_1}s+\underbrace{(\Delta_{00}^1)^{2\lambda_2}\Delta_{01}^2}_{\Upsilon_2}s^{\lambda_1\lambda_2}+\mathcal{F}^\infty_{\ell_7}(\mu_0)\bigr), \vspace{0.2cm} \\
            
            \qquad=s^{\lambda_1\lambda_2}\bigl(\Upsilon_0+(\Upsilon_1+\Upsilon_2s^{-\alpha})s+\mathcal{F}^\infty_{\ell_6}(\mu_0)\bigr) \vspace{0.2cm} \\
            \qquad\quad=s^{\lambda_1\lambda_2}\bigl(\Upsilon_0+\bigl(\Upsilon_1+\Upsilon_2(1+\alpha\omega(s;\alpha)\bigr)s+\mathcal{F}^\infty_{\ell_7}(\mu_0)\bigr) \vspace{0.2cm} \\

            \qquad\qquad =s^{\lambda_1\lambda_2}\bigl(\Upsilon_0+\Upsilon_\omega s+\mathcal{F}^\infty_{\ell_7}(\mu_0)\bigr),
        \end{array}
        \end{equation}
    with $\ell_7\in(1,\min\{\ell_1,\ell_4,\lambda_1^0\lambda_2^0+1\})=(1,\min\{\lambda_1^0,2\})$ and the last third due to \eqref{8}. The lemma now follows from \eqref{6}, \eqref{7} and \eqref{9}.
\end{proof}

\begin{obs}
    Under the hypothesis of Lemma~\ref{Lemma5} we observe that the compensator $\omega(s;\alpha)$ appearing when $\lambda_1^0\lambda_2^0=1$ is a compact way to write $D_2\circ D_1$ in this case. More precisely suppose $\lambda_1^0\lambda_2^0=1$ and observe that given $\lambda_1\approx\lambda_1^0$ and $\lambda_2\approx\lambda_2^0$ we have $\alpha=0$ if and only if $\lambda_1\lambda_2=1$. Moreover if $\lambda_1\lambda_2\neq1$ then it follows from \eqref{8} that $(1+\alpha\omega(s;\alpha))s=s^{\lambda_1\lambda_2}$. Replacing this at~\eqref{9} we obtain that if $\lambda_1^0\lambda_1^0=1$, then
        \begin{equation}\label{00}
        D_2\circ D_1(s;\mu)=\left\{\begin{array}{ll}
            s^{\lambda_1\lambda_2}\bigl(\Upsilon_0+\Upsilon_1 s+\Upsilon_2s^{\lambda_1\lambda_2}+\mathcal{F}^\infty_{\ell'}(\mu_0)\bigr), & \text{if } \lambda_1\lambda_2>1, \vspace{0.2cm} \\
            
            s\cdot\bigl(\Upsilon_0+(\Upsilon_1+\Upsilon_2)s+\mathcal{F}^\infty_{\ell'}(\mu_0)\bigr), & \text{if } \lambda_1\lambda_2=1, \vspace{0.2cm} \\
            
            s^{\lambda_1\lambda_2}\bigl(\Upsilon_0+\Upsilon_2s^{\lambda_1\lambda_2}+\Upsilon_1s+\mathcal{F}^\infty_{\ell'}(\mu_0)\bigr), & \text{if } \lambda_1\lambda_2<1. 
        \end{array}\right.
        \end{equation}
    That is, the next term after the leading one depend on the sign of $1-\lambda_1\lambda_2$. Since the initial condition satisfies $\lambda_1^0\lambda_2^0=1$ we have that the explicit expression of $D_2\circ D_1$ can change with arbitrarily small perturbations at the initial condition. Therefore, to understand the regularity of the compensator $\omega(s;\mu)$ helps to understand the regularity of $D_2\circ D_1$ when interchanging among the explicit expressions given at~\eqref{00}. To this end, we refer to Lemma~$A.3$ and Corollary~$A.7$ of \cite{MarVilDulacLocal}. 
\end{obs}

\begin{lem}\label{Lemma6}
    Let $\{X_\mu\}_{\mu\in\Lambda}$ be a smooth family of planar smooth vector fields having a persistent polycycle $\Gamma$ with hyperbolic saddles $p_1$ and $p_2$. Let $\mu_0\in\Lambda$ be such that $\lambda_1^0<1$ and $\lambda_2^0>1$. Then for any given $\ell\in(\lambda_1^0,\min\{\lambda_1^0\lambda_2^0,2\lambda_1^0,1\})$ it holds
        \[D_2\circ D_1(s;\mu)=s^{\lambda_1\lambda_2}\bigl(\Upsilon_0+\Upsilon_3s^{\lambda_1}+\mathcal{F}^\infty_\ell(\mu_0)\bigr),\]
    where $\Upsilon_0=(\Delta_{00}^1)^{\lambda_2}\Delta_{00}^2$ and $\Upsilon_3=\lambda_2(\Delta_{00}^1)^{\lambda_2-1}\Delta_{00}^2\Delta_{01}^1+(\Delta_{00}^1)^{\lambda_2+1}\Delta_{10}^2$.
\end{lem}

\begin{proof}
    Let $\ell\in(\lambda_1^0,\min\{\lambda_1^0\lambda_2^0,2\lambda_1^0,1\})$. Since $\lambda_1^0<1$ and $\lambda_2^0>1$, from Theorem~\ref{Teo:Dulacmap} we have that
        \[D_1(s;\mu)=s^{\lambda_1}\bigl(\Delta_{00}^1+\Delta_{01}^1s^{\lambda_1}+\mathcal{F}^\infty_{\ell_1}(\mu_0)\bigr), \quad D_2(s;\mu)=s^{\lambda_2}\bigl(\Delta_{00}^2+\Delta_{10}^2s+\mathcal{F}^\infty_{\ell_2}(\mu_0)\bigr),\]
    for any given $\ell_1\in(\lambda_1^0,\min\{2\lambda_1^0,1\})$ and $\ell_2\in(1,\min\{\lambda_2^0,2\})$. Similarly to the previous cases we observe that
        \[\begin{array}{l}
            D_2\circ D_1(s;\mu)=s^{\lambda_1\lambda_2}\bigl(\Delta_{00}^1+\Delta_{01}^1s^{\lambda_1}+\mathcal{F}^\infty_{\ell_1}(\mu_0)\bigr)^{\lambda_2}\cdot \vspace{0.2cm} \\
            
            \quad\cdot\bigl(\Delta_{00}^2+\Delta_{10}^2s^{\lambda_1}\bigl(\Delta_{00}^1+\Delta_{01}^1s^{\lambda_1}+\mathcal{F}^\infty_{\ell_1}(\mu_0)\bigr)+\mathcal{F}^\infty_{\ell_3}(\mu_0)\bigr) \vspace{1cm} \\

            \qquad=s^{\lambda_1\lambda_2}\bigl((\Delta_{00}^1)^{\lambda_2}+\lambda_2(\Delta_{00}^1)^{\lambda_2-1}\Delta_{01}^1s^{\lambda_1}+\mathcal{F}^\infty_{\ell_1}(\mu_0)\bigr)\cdot \vspace{0.2cm} \\
            \qquad\quad\cdot\bigl(\Delta_{00}^2+\Delta_{00}^1\Delta_{10}^2s^{\lambda_1}+\mathcal{F}^\infty_{\ell_4}(\mu_0)\bigr) \vspace{1cm} \\

            \qquad\quad=s^{\lambda_1\lambda_2}\bigl(\Delta_{00}^2(\Delta_{00}^1)^{\lambda_2}+\bigl(\lambda_2\Delta_{00}^2(\Delta_{00}^1)^{\lambda_2-1}\Delta_{01}^1+(\Delta_{00}^1)^{\lambda_2+1}\Delta_{10}^2\bigr)s^{\lambda_1}+\mathcal{F}^\infty_{\ell_5}(\mu_0)\bigr),
        \end{array}\]
        with $\ell_3\in(\lambda_1^0,\min\{\lambda_1^0\lambda_2^0,2\lambda_1^0\})$, $\ell_4\in(\lambda_1^0,\min\{2\lambda_1^0,\lambda_1^0+\ell_1,\ell_3\})=(\lambda_1^0,\min\{\lambda_1^0\lambda_2^0,2\lambda_1^0\})$ and 
            \[\ell_5\in(\lambda_1^0,\min\{\ell_1,\ell_4,2\lambda_1\})=(\lambda_1^0,\min\{\lambda_1^0\lambda_2^0,2\lambda_1^0,1\}).\]
        In particular we can take $\ell_5=\ell$.
\end{proof}

\subsection{Inverse of a Dulac map}

Note that if $\{X_\mu\}_{\mu\in\Lambda}$ is a smooth family of planar smooth vector fields having a hyperbolic saddle $p=p(\mu)$ with hyperbolicity ratio $\lambda(\mu)$, then its Dulac map $D(s;\mu)$ has a well defined inverse $D^{-1}(s;\mu)$ which happens to be the Dulac map of $p$ in relation to the family $\{-X_\mu\}_{\mu\in\Lambda}$, with hyperbolicity ratio $\lambda(\mu)^{-1}$. With this knowledge, we can use the previous lemmas to obtain a formula for the first coefficients of $D^{-1}$ in function of the coefficients of $D$.

\begin{lem}
    Let $\{X_\mu\}_{\mu\in\Lambda}$ be a smooth family of planar smooth vector fields having a hyperbolic saddle $p$. Set $\rho=\lambda^{-1}$, $\rho^0=(\lambda^0)^{-1}$ and 
        \[D(s;\mu)=s^\lambda\bigl(\Delta_{00}+\mathcal{F}^\infty_\ell(\mu_0)\bigr), \quad D^{-1}(s;\mu)=s^{\rho}\bigl(\Omega_{00}+\mathcal{F}^\infty_\eta(\mu_0)\bigr),\]
    with $\ell\in(0,\min\{\lambda^0,1\})$, $\eta\in(0,\min\{\rho^0,1\})$. Then $\Omega_{00}=\Delta_{00}^{-\rho}$.
\end{lem}

\begin{proof}
    On one hand we have from Lemma~\ref{Lemma2} that
        \[D^{-1}\circ D(s;\mu)=s\bigl(\Upsilon_0+\mathcal{F}^\infty_{\ell'}(\mu_0)\bigr),\]
    for any given $\ell'\in(0,\min\{\lambda^0,\lambda^0\rho^0\})=(0,\min\{\lambda^0,1\})$,
    where $\Upsilon_0=\Delta_{00}^\rho\Omega_{00}$. On the other hand we have $D^{-1}\circ D(s;\mu)=s$. In particular it follows that $\Upsilon_0=1$, from which we obtain $\Omega_{00}=\Delta_{00}^{-\rho}$.
\end{proof}

\begin{lem}\label{Lemma7}
    Let $\{X_\mu\}_{\mu\in\Lambda}$ be a smooth family of planar smooth vector fields having a hyperbolic saddle $p$. Let $\mu_0\in\Lambda$ be such that $\lambda^0<1$ and denote
        \[D(s;\mu)=s^\lambda\bigl(\Delta_{00}+\Delta_{01}s^\lambda+\mathcal{F}^\infty_\ell(\mu_0)\bigr), \quad D^{-1}(s;\mu)=s^{\rho}\bigl(\Omega_{00}+\Omega_{10}s+\mathcal{F}^\infty_\eta(\mu_0)\bigr),\]
    with $\ell\in(\lambda^0,\min\{2\lambda^0,1\})$, $\eta\in(1,\min\{\rho^0,2\})$. Then
        \[\Omega_{00}=\Delta_{00}^{-\rho}, \quad \Omega_{10}=-\rho\Delta_{00}^{-(2+\rho)}\Delta_{01},\]
    where $\rho=\lambda^{-1}$ and $\rho^0=(\lambda^0)^{-1}$.
\end{lem}

\begin{proof}
    On the one hand we have from Lemma~\ref{Lemma6} that
        \[D^{-1}\circ D(s;\mu)=s\bigl(\Upsilon_0+\Upsilon_3s^\lambda+\mathcal{F}^\infty_{\ell'}(\mu_0)\bigr),\]
    for any given $\ell'\in(\lambda^0,\min\{2\lambda^0,1\})$, where
        \[\Upsilon_0=\Delta_{00}^\rho\Omega_{00}, \quad \Upsilon_3=\rho\Delta_{00}^{\rho-1}\Omega_{00}\Delta_{01}+\Delta_{00}^{\rho+1}\Omega_{10}.\]
    On the other hand we have $D^{-1}\circ D(s;\mu)=s$. In particular it follows that $\Upsilon_0=1$ and $\Upsilon_3=0$. From the former we obtain $\Omega_{00}=\Delta_{00}^{-\rho}$. Replacing this at the latter we obtain the formula for $\Omega_{10}$.
\end{proof}

\begin{lem}\label{Lemma8}
    Let $\{X_\mu\}_{\mu\in\Lambda}$ be a smooth family of planar smooth vector fields having a hyperbolic saddle $p$. Let $\mu_0\in\Lambda$ be such that $\lambda^0>1$ and denote
        \[D(s;\mu)=s^\lambda\bigl(\Delta_{00}+\Delta_{10}s+\mathcal{F}^\infty_\ell(\mu_0)\bigr), \quad D^{-1}(s;\mu)=s^{\rho}\bigl(\Omega_{00}+\Omega_{01}s^\rho+\mathcal{F}^\infty_\eta(\mu_0)\bigr),\]
    with $\ell\in(1,\min\{\lambda^0,2\})$, $\eta\in(\rho^0,\min\{2\rho^0,1\})$. Then
        \[\Omega_{00}=\Delta_{00}^{-\rho}, \quad \Omega_{01}=-\rho\Delta_{00}^{-(1+2\rho)}\Delta_{10},\]
    where $\rho=\lambda^{-1}$ and $\rho^0=(\lambda^0)^{-1}$.
\end{lem}

\begin{proof}
    Similarly to Lemma~\ref{Lemma7}, it follows from Lemma~\ref{Lemma6} that
        \[D\circ D^{-1}(s;\mu)=s\bigl(\Upsilon_0+\Upsilon_3s^\rho+\mathcal{F}^\infty_{\ell'}(\mu_0)\bigr),\]
    for any given $\ell'\in(\rho^0,\min\{2\rho^0,1\})$, where
        \[\Upsilon_0=\Omega_{00}^\lambda\Delta_{00}, \quad \Upsilon_3=\lambda\Omega_{00}^{\lambda-1}\Delta_{00}\Omega_{01}+\Omega_{00}^{\lambda+1}\Delta_{10}.\]
    The result now follows by observing that $\Upsilon_0=1$ and $\Upsilon_3=0$.
\end{proof}

\subsection{Coefficients of the return map}

The following results present explicit formulas for the first coefficients in the asymptotic expansion of the return map $\mathscr{R}(s;\mu)$.

\begin{propo}\label{Propo:Return1-+}
 Let $\{X_\mu\}_{\mu\in\Lambda}$ be a smooth family of planar smooth vector fields having a persistent polycycle $\Gamma$ with hyperbolic saddles $p_1,\dots,p_m,p_{m+1},\dots,p_n$. Let $\mu_0\in\Lambda$ be such that $\lambda_i(\mu_0)<1$ for $i\in\{1,\dots,m\}$ and $\lambda_i(\mu_0)>1$ for $i\in\{m+1,\dots,n\}$. Then the return map of $\Gamma$ is given by
\begin{equation}\label{eq:Return2-+}.
        \mathscr{R}(s;\mu)=s^{r(\mu)}\bigl(A_{1,n}+\mathcal{A}s^{\Lambda_{0,m}}+\mathcal{F}^\infty_\ell(\mu_0)\bigr),
    \end{equation}
    for any given $\ell\in(\Lambda_{0,m}^0,\min\{r(\mu_0),2\Lambda_{0,m}^0,1\})$,
    where $\mathcal{A}=\Lambda_{m,n}A_{1,m}A_{1,n}(S_1^{m+1}-S_2^m)$.
\end{propo}
\begin{proof}
From Corollary~\ref{Coro2} we have that
    \begin{equation}\label{17}
        D_m\circ\ldots\circ D_1(s;\mu)=s^{\Lambda_{0,m}}\bigl(A_{1,m}+C_{1,m}s^{\Lambda_{0,m}}+\mathcal{F}^\infty_{\ell_1}(\mu_0)\bigr),
    \end{equation}
    where $\ell_1\in(\Lambda_{0,m}^0,\min\{\Lambda_{0,m-1}^0,2\Lambda_{0,m}^0\})$ and
        \[C_{j,k}=A_{j,k-1}^{2\lambda_k}\Delta_{01}^k, \quad A_{j,k}=\prod_{i=j}^{k}(\Delta_{00}^{i})^{\Lambda_{i,k}}, \quad \Lambda_{i,k}=\prod_{j=i+1}^{k}\lambda_j,\;\Lambda_{kk}=1.\]
    Moreover from Corollary~\ref{Coro1} we have that
    \begin{equation}\label{18}
        D_n\circ\ldots\circ D_{m+1}(s;\mu)=s^{\Lambda_{m,n}}\bigl(A_{m+1,n}+B_{m+1,n}s+\mathcal{F}^\infty_{\ell_{m+1}}(\mu_0)\bigr),
    \end{equation}
    with $\ell_{m+1}\in(1,\min\{\lambda_{m+1}^0,2\})$, where 
        \[B_{j,k}=\Lambda_{j,k}\frac{\Delta_{10}^j}{\Delta_{00}^j}A_{j,k}, \quad A_{j,k}=\prod_{i=j}^{k}(\Delta_{00}^{i})^{\Lambda_{i,k}}, \quad \Lambda_{i,k}=\prod_{j=i+1}^{k}\lambda_j,\;\Lambda_{kk}=1.\]
    Since \eqref{17} and \eqref{18} are of \emph{Dulac-type} (i.e. it has similar expression), it follows mutatis mutandis from Lemma~\ref{Lemma6} that
    \[\begin{array}{ll}
        \mathscr{R}(s;\mu) &= (D_n\circ\ldots\circ D_{m+1})\circ(D_m\circ\ldots\circ D_1)(s;\mu) \vspace{0.2cm} \\
        &=s^{r}\big(A_{1,m}^{\Lambda_{m,n}}A_{m+1,n}+\mathcal{A}s^{\Lambda_{0,m}}+\mathcal{F}^\infty_{\ell}(\mu_0)\bigr) \vspace{0.2cm} \\ 

        &=s^{r}\big(A_{1,n}+\mathcal{A}s^{\Lambda_{0,m}}+\mathcal{F}^\infty_{\ell}(\mu_0)\bigr),
    \end{array}\]
    where
        \[\begin{array}{ll}
            \mathcal{A} &\displaystyle= \Lambda_{m,n}A_{1,n}\frac{C_{1,m}}{A_{1,m}}+A_{1,m}^{\Lambda_{m,n}+1}B_{m+1,n} \vspace{0.2cm} \\
            
            &\displaystyle= \Lambda_{m,n}A_{1,n}\frac{A_{1,m-1}^{2\lambda_m}}{A_{1,m}}\Delta_{0,1}^m+\Lambda_{m+1,n}A_{1,m}^{\Lambda_{m,n}}A_{m+1,n}A_{1,m}\frac{\Delta_{10}^{m+1}}{\Delta_{00}^{m+1}} \vspace{0.2cm} \\

            &\displaystyle=-\Lambda_{m,n}A_{1,n}\frac{(A_{1,m-1}^{\lambda_m}\Delta_{00}^m)^2}{A_{1,m}}S_2^m+\lambda_m\Lambda_{m+1,n}A_{1,n}A_{1,m}S_1^{m+1} \vspace{0.2cm} \\

            &\displaystyle= \Lambda_{m,n}A_{1,n}A_{1,m}(S_1^{m+1}-S_2^m),
    \end{array}\]
    and $\ell\in(\Lambda_{0,m}^0,\min\{r(\mu_0),2\Lambda_{0,m}^0,1\})$.
\end{proof}

\begin{propo}\label{Propo:Return1+-}
    Let $\{X_\mu\}_{\mu\in\Lambda}$ be a smooth family of planar smooth vector fields having a persistent polycycle $\Gamma$ with hyperbolic saddles $p_1,\dots,p_m,p_{m+1},\dots,p_n$. Let $\mu_0\in\Lambda$ be such that $\lambda_i(\mu_0)>1$ for $i\in\{1,\dots,m\}$ and $\lambda_i(\mu_0)<1$ for $i\in\{m+1,\dots,n\}$. Then the return map of $\Gamma$ is given by
    \begin{equation}\label{eq:Return2+-}
        \mathscr{R}(s;\mu)=\left\{\begin{array}{ll}
            s^{r(\mu)}(A_{1,n}+\mathcal{B}s+\mathcal{F}^\infty_\ell(\mu_0)), & \text{if } r(\mu_0)>1, \vspace{0.2cm} \\
            
            s^{r(\mu)}(A_{1,n}+\mathcal{A}_\omega s+\mathcal{F}^\infty_{\ell'}(\mu_0)), & \text{if } r(\mu_0)=1, \vspace{0.2cm} \\
            
            s^{r(\mu)}(A_{1,n}+\mathcal{C} s^{r(\mu)}+\mathcal{F}^\infty_{\ell''}(\mu_0)), & \text{if } r(\mu_0)<1, 
        \end{array}\right.
    \end{equation}
    for any given 
        \[\ell\in(1,\min\{r(\mu_0),2\}), \quad \ell'\in(1,\min\{\Lambda_{0,m}^0,2\}), \quad \ell''\in(r(\mu_0),\min\{2r(\mu_0),1\});\]
    where 
        \[\mathcal{B}=r(\mu)A_{1,n}S_1^1, \quad \mathcal{C}=-A_{1,n}^2S_2^n, \quad \mathcal{A}_\omega=\mathcal{B}+\bigl(1+\alpha\omega(s;\alpha)\bigr)\mathcal{C},\]
    $\alpha=1-r(\mu)$, $\lambda_i^0=\lambda_i(\mu_0)$ and $r(\mu)=\lambda_1(\mu)\dots\lambda_n(\mu)$.
\end{propo}
\begin{proof}
Given $\ell_1\in(1,\min\{\lambda_1^0,2\})$ it follows from Corollary~\ref{Coro1} that 
    \begin{equation}\label{10}
        D_m\circ\dots\circ D_1(s;\mu)=s^{\Lambda_{0,m}}\bigl(A_{1,m}+B_{1,m} s+\mathcal{F}^{\infty}_{\ell_1}(\mu_0)\bigr),
    \end{equation}
    where we recall that 
        \[B_{j,k}=\Lambda_{j,k}\frac{\Delta_{10}^j}{\Delta_{00}^j}A_{j,k}, \quad A_{j,k}=\prod_{i=j}^{k}(\Delta_{00}^{i})^{\Lambda_{i,k}}, \quad \Lambda_{i,k}=\prod_{j=i+1}^{k}\lambda_j,\;\Lambda_{kk}=1.\]
    From Corollary~\ref{Coro2} we have that
    \begin{equation}\label{10x}
        D_n\circ\dots\circ D_{m+1}(s;\mu)=s^{\Lambda_{m,n}}\bigl(A_{m+1,n}+C_{m+1,n} s^{\Lambda_{m,n}}+\mathcal{F}^{\infty}_{\ell_n}(\mu_0)\bigr),
    \end{equation}
    where $\ell_n\in(\Lambda_{m,n}^0,\min\{\Lambda_{m,n-1}^0,2\Lambda_{m,n}^0\})$ and $C_{j,k}=A_{j,k-1}^{2\lambda_k}\Delta_{01}^k$.
    
    Since \eqref{10} and \eqref{10x} are of Dulac-type and $r(\mu_0)=1$, it follows mutatis mutandis from Lemma~\ref{Lemma5} that the first return map 
        \[\mathscr{R}(s;\mu)=(D_n\circ\dots\circ D_{m+1})\circ (D_m\circ\dots\circ D_1)(s;\mu),\]
    is given by 
    \begin{equation}\label{11}
        \mathscr{R}(s;\mu)=\left\{\begin{array}{ll}
            s^{r(\mu)}\bigl(A_{1,n}+\mathcal{B}s+\mathcal{F}^\infty_\ell(\mu_0)\bigr), & \text{if } r(\mu_0)>1, \vspace{0.2cm} \\
            
            s^{r(\mu)}\bigl(A_{1,n}+\mathcal{A}_\omega s+\mathcal{F}^\infty_{\ell'}(\mu_0)\bigr), & \text{if } r(\mu_0)=1, \vspace{0.2cm} \\
            
            s^{(\mu)}\bigl(A_{1,n}+\mathcal{C} s^{r(\mu)}+\mathcal{F}^\infty_{\ell''}(\mu_0)\bigr), & \text{if } r(\mu_0)<1, 
        \end{array}\right.
    \end{equation}
    for any given 
        \[\ell\in(1,\min\{r(\mu_0),2\}), \quad \ell'\in(1,\min\{\Lambda_{0,m}^0,2\}), \quad \ell''\in(r(\mu_0),\min\{2r(\mu_0),1\});\]
    where
        \begin{equation}\label{12}
        \begin{array}{l}
            \displaystyle \mathcal{A}_\omega=\mathcal{B}+\bigl(1+\alpha\omega(s;\alpha)\bigr)\mathcal{C}, \vspace{0.2cm} \\
            
            \displaystyle \mathcal{B}=\Lambda_{m,n}A_{1,m}^{\Lambda_{m,n}}A_{m+1,n}\frac{B_{1,m}}{A_{1,m}}=\Lambda_{1,m}\Lambda_{m,n}A_{1,n}\frac{\Delta_{10}^1}{\Delta_{00}^1}=r(\mu)A_{1,n}S_1^1, \vspace{0.2cm} \\

            \displaystyle \mathcal{C}=A_{1,m}^{2\Lambda_{m,n}}C_{m+1,n}=A_{1,m}^{2\Lambda_{m,n}}A_{m+1,n-1}^{2\lambda_n}\Delta_{01}^n=-\bigl(A_{1,m}^{\Lambda_{m,n}}A_{m+1,n-1}^{\lambda_n}\Delta_{00}^n\bigr)^2S_2^n=-A_{1,n}^2S_2^n.
        \end{array}
        \end{equation}
    and $\alpha=1-r(\mu)$. The results now follows from \eqref{11} and \eqref{12}.
\end{proof}

\begin{obs}\label{R2}
    Although the expressions of the return map in Proposition \ref{Propo:Return1-+} and \ref{Propo:Return1+-} are different, the situation they described is the same. Indeed, under the hypothesis of Proposition \ref{Propo:Return1+-}, one can always relabel the corners of $\Gamma$ so that the first saddles begin with $\lambda_i(\mu_0)<1$, see Figure~\ref{Fig2}. 
\end{obs}

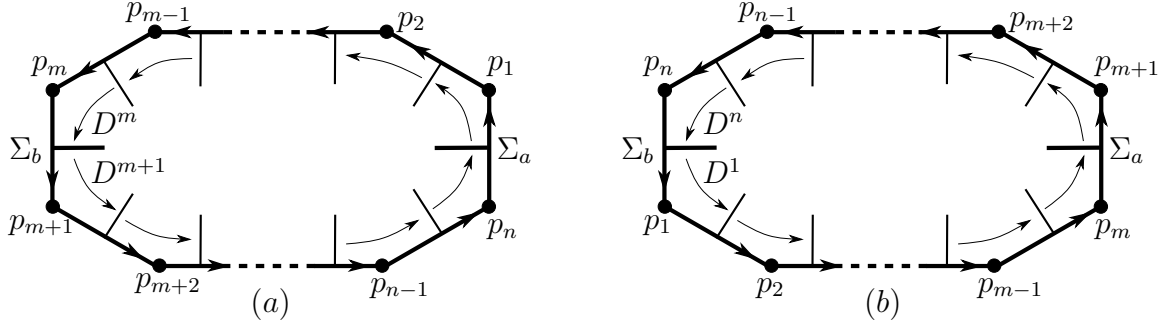
\begin{figure}[ht]
	\begin{center}
		\begin{minipage}{8cm}
			\begin{center} 
			 	\begin{overpic}[width=6cm]{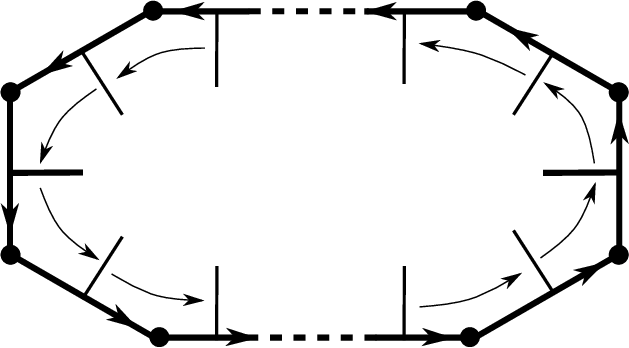} 
				    \put(98,45){$p_1$}
                        \put(78,55){$p_2$}
                        \put(18,57){$p_{m-1}$}
                        \put(-3,45){$p_m$}
                        \put(-8,10){$p_{m+1}$}
                        \put(20,-3){$p_{m+2}$}
                        \put(72,-4){$p_{n-1}$}
                        \put(98,9){$p_n$}
                        \put(10,31){$D^m$}
                       \put(10,20){$D^{m+1}$}
                       \put(100,25.5){$\Sigma_a$}
                       \put(-8,25.5){$\Sigma_b$}
			 	\end{overpic}
				
		  		$(a)$
			\end{center}
		\end{minipage}
		\begin{minipage}{8cm}
			\begin{center} 
			 	\begin{overpic}[width=6cm]{Fig3.eps} 
			     	\put(97,45){$p_{m+1}$}
                        \put(78,55){$p_{m+2}$}
                        \put(18,57){$p_{n-1}$}
                        \put(-3,45){$p_n$}
                        \put(-3,10){$p_1$}
                        \put(22,-4){$p_2$}
                        \put(71,-4){$p_{m-1}$}
                        \put(97,9){$p_m$}
                        \put(10,31){$D^n$}
                       \put(10,20){$D^1$}
                       \put(100,25.5){$\Sigma_a$}
                       \put(-8,25.5){$\Sigma_b$}
		      	\end{overpic}
				
		  		$(b)$
		  	\end{center}
		\end{minipage}
	\end{center}
	\caption{Illustration of the equivalence between the hypothesis of $(a)$ Proposition~\ref{Propo:Return1-+} (with the indexation starting at $\Sigma_a$) and $(b)$ Proposition~\ref{Propo:Return1+-} (with the indexation starting at $\Sigma_b$). Observe that regardless of the indexation, the only Dulac maps whose the non-leader term appears in the expressions of~\eqref{eq:Return2-+} and~\eqref{eq:Return2+-} are those defined near the transversal $\Sigma_b$.}\label{Fig2}
\end{figure}

\section{The displacement map of a persistent polycycle}\label{Sec:displacement1}

Consider a persistent polycycle $\Gamma$ of a smooth family $\{X_\mu\}_{\mu\in\Lambda}$ of planar smooth vector fields with hyperbolic saddles $p_1,\dots,p_m,p_{m+1},\dots p_n$. Throughout the text of the present paper, we dealt with the return map, as its isolated fixed points correspond to limit cycles in a neighborhood of $\Gamma$. However, there are particular situations (see, for instance~\cite{MarVilKolmogorov}) where it is more convenient to work with the difference of the transition maps from $p_1$ to $p_m$ and from $p_{n}$ to $p_{m+1}$, with the last one following the solution of $\{-X_\mu\}_{\mu\in\Lambda}$. More precisely, the following displacement map:
    \begin{equation}\label{23}
        \mathscr{D}(s;\mu)=D_m\circ\dots\circ D_1(s;\mu)-D_{m+1}^{-1}\circ\dots\circ D_{n}^{-1}(s;\mu).
    \end{equation}
See Figure~\ref{Fig3}. 
\begin{figure}[ht]
	\begin{center}
		\begin{minipage}{8cm}
			\begin{center} 
			 	\begin{overpic}[width=6cm]{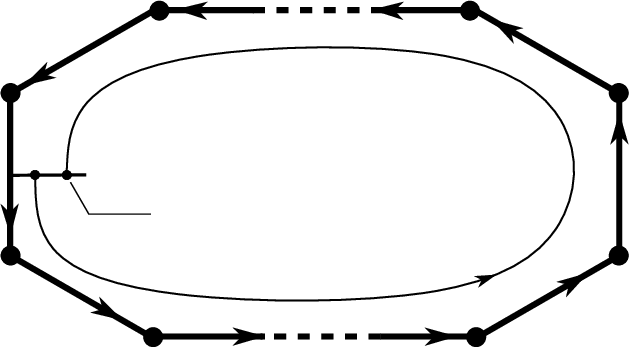} 
                        \put(97,45){$p_{m+1}$}
                        \put(78,55){$p_{m+2}$}
                        \put(18,57){$p_{n-1}$}
                        \put(-3,45){$p_n$}
                        \put(-3,10){$p_1$}
                        \put(22,-4){$p_2$}
                        \put(71,-4){$p_{m-1}$}
                        \put(97,9){$p_m$}
                       \put(-8,25.5){$\Sigma_b$}
                       \put(4.5,29){$s$}
                       \put(25,19){$\mathscr{R}(s)$}
			 	\end{overpic}
				
		  		$(a)$
			\end{center}
		\end{minipage}
		\begin{minipage}{8cm}
			\begin{center} 
			 	\begin{overpic}[width=6cm]{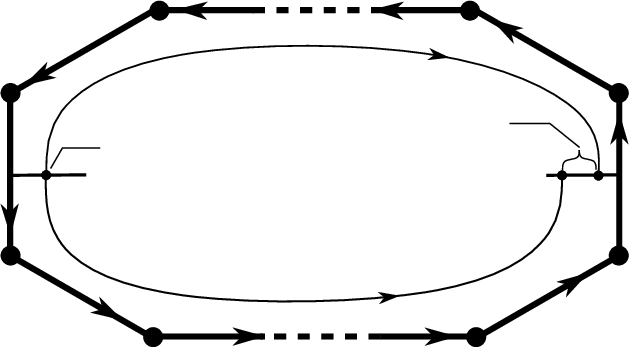} 
                        \put(97,45){$p_{m+1}$}
                        \put(78,55){$p_{m+2}$}
                        \put(18,57){$p_{n-1}$}
                        \put(-3,45){$p_n$}
                        \put(-3,10){$p_1$}
                        \put(22,-4){$p_2$}
                        \put(71,-4){$p_{m-1}$}
                        \put(97,9){$p_m$}
                       \put(100,25.5){$\Sigma_a$}
                       \put(-8,25.5){$\Sigma_b$}
                       \put(16.5,30.25){$s$}
                       \put(65,34){$\mathscr{D}(s)$}
		      	\end{overpic}
				
		  		$(b)$
		  	\end{center}
		\end{minipage}
	\end{center}
	\caption{Illustration of $(a)$ the first return map $\mathscr{R}$ and $(b)$ the displacement map $\mathscr{D}$.}\label{Fig3}
\end{figure}
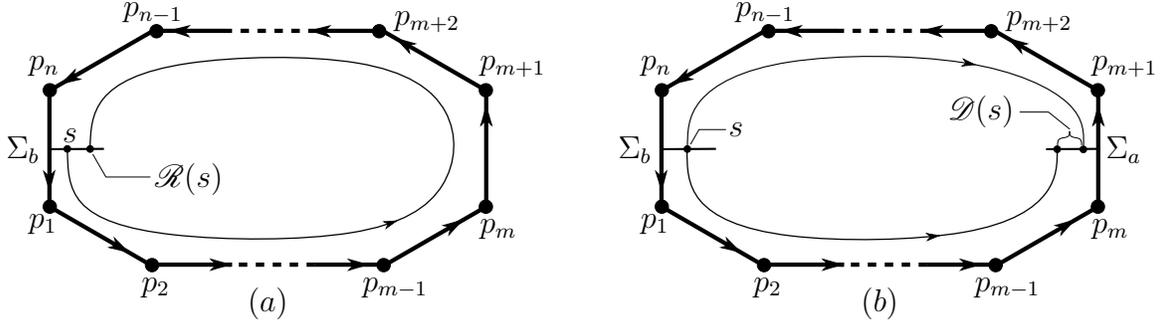
Observe that both approaches are equivalent, since
\begin{eqnarray*}
    \mathscr{R}(s;\mu)=s&\iff& D_n\circ\dots\circ D_{m+1}\circ D_m\circ\dots\circ D_1(s;\mu)=Id(s)\\
    &\iff&  D_m\circ\dots\circ D_1(s;\mu) = \left(D_n\circ\dots\circ D_{m+1}(s;\mu)\right)^{-1}\nonumber\\
    &\iff& \mathscr{D}(s;\mu)=0.
\end{eqnarray*}
For these situations, similar to Propositions~\ref{Propo:Return1-+} and~\ref{Propo:Return1+-}, the next result determines the leading terms of the displacement map $\mathscr{D}(s;\mu)$.

\begin{propo}\label{Propo:displacement+-}
    Let $\{X_\mu\}_{\mu\in\Lambda}$ be a smooth family of planar smooth vector fields having a persistent polycycle $\Gamma$ with hyperbolic saddles $p_1,\dots,p_m,p_{m+1},\dots,p_n$. Let $\mu_0\in\Lambda$ be such that $\lambda_i(\mu_0)>1$ for $i\in\{1,\dots,m\}$ and $\lambda_i(\mu_0)<1$ for $i\in\{m+1,\dots,n\}$. Then the displacement map of $\Gamma$ is given by
    \begin{equation}
        \mathscr{D}(s;\mu)=\left\{\begin{array}{ll}
            A_{1,m}s^{\Lambda_{0,m}}-A_{m+1,n}^*s^{\Lambda_{m,n}^{-1}}+\mathcal{F}^\infty_{\ell}(\mu_0), & \text{if } r(\mu_0)\neq1, \vspace{0.2cm} \\
            
            s^{\Lambda_{m,n}^{-1}}U(s;\mu)                      \bigl(\Psi_1\omega(s;\alpha)+\Psi_2+\Psi_3s+\mathcal{F}^\infty_{\ell'}(\mu_0)\bigr), & \text{if } r(\mu_0)=1,
        \end{array}\right.
    \end{equation}
    for any given 
        \[\ell\in(\max\{\Lambda_{0,m}^0,(\Lambda_{m,n}^0)^{-1}\},\min\{\Lambda_{0,m}^0+1,(\Lambda_{m,n}^0)^{-1}+1\}), \quad \ell'\in(1,\min\{\lambda_1^0,(\lambda_n^0)^{-1},2\}),\]
    where 
    \begin{equation}
        \Psi_1=\alpha A_{1,m}, \quad \Psi_2=A_{1,m}-A_{m+1,n}^*, \quad \Psi_3=A_{m+1,n}^*\bigl(\Lambda_{0,m}S_1^1-\Lambda_{m,n}^{-1}S_2^n),
    \end{equation}
    $\alpha=\Lambda_{m,n}^{-1}-\Lambda_{0,m}$ and $U(s;\mu)=1+\Lambda_{0,m}S_1^1s+\mathcal{F}^\infty_{\ell'}(\mu_0)$.
\end{propo}
\begin{proof}
    For $i\in\{1,\dots,m\}$ let
        \[D_i(s;\mu)=s^{\lambda_i}\bigl(\Delta_{00}^i+\Delta_{10}^i s+\mathcal{F}^\infty_{\ell_i}(\mu_0)\bigr),\]
    with $\ell_i\in(1,\min\{\lambda_i,2\})$. It follows from Corollary~\ref{Coro1} that
    \begin{equation}\label{13}
        D_m\circ\dots\circ D_1(s;\mu)=s^{\Lambda_{0,m}}\bigl(A_{1,m}+B_{1,m}s+\mathcal{F}^\infty_{\ell_1}(\mu_0)\bigr),
    \end{equation}
    where we recall that,
      \[B_{j,k}=\Lambda_{j,k}\frac{\Delta_{10}^j}{\Delta_{00}^j}A_{j,k}, \quad A_{j,k}=\prod_{i=j}^{k}(\Delta_{00}^{i})^{\Lambda_{i,k}}, \quad \Lambda_{i,k}=\prod_{j=i+1}^{k}\lambda_j,\;\Lambda_{kk}=1.\]   
    For $i\in\{m+1,\dots,n\}$ let
        \[D_i(s;\mu)=s^{\lambda_i}\bigl(\Delta_{00}^i+\Delta_{01}^is^{\lambda_i}+\mathcal{F}^\infty_{\ell_i}(\mu_0)\bigr), \quad D_i^{-1}(s;\mu)=s^{\lambda_i^{-1}}\bigl(\Omega_{00}^i+\Omega_{10}^is+\mathcal{F}^\infty_{\eta_i}(\mu_0)\bigr),\]
    with $\ell_i\in(\lambda_i^0,\min\{2\lambda_i^0,1\})$ and $\eta_i\in(1,\min\{(\lambda_i^0)^{-1},2\})$. From Corollary~\ref{Coro1} we have 
    \begin{equation}\label{13x}
        D_{m+1}^{-1}\circ\dots\circ D_n^{-1}(s;\mu)=s^{\Lambda_{m,n}^{-1}}\bigl(A_{m+1,n}^*+B_{m+1,n}^*s+\mathcal{F}^\infty_{\eta_n}(\mu_0)\bigr),
    \end{equation}
    where
    \begin{equation}\label{15}
    \begin{array}{l}
        \displaystyle A_{j,k}^*=\prod_{i=0}^{k-j}(\Omega_{00}^{k-i})^{\Lambda_{j-1,k-1-i}^{-1}} = \prod_{i=0}^{k-j}(\Delta_{00}^{k-i})^{-\Lambda_{j-1,k-i}^{-1}}, \vspace{0.2cm} \\

        \displaystyle B_{j,k}^*=\Lambda_{j-1,k-1}^{-1}\frac{\Omega_{10}^k}{\Omega_{00}^k}A_{j,k}^*=-\Lambda_{j-1,k}^{-1}\frac{\Delta_{01}^k}{(\Delta_{00}^k)^2}A_{j,k}^*,
    \end{array}
    \end{equation}
    with the last equality on both lines following from Lemma~\ref{Lemma7}. Observe that $r(\mu)\neq1$ if, and only if, $\Lambda_{m,n}^{-1}\neq\Lambda_{0,m}$. Therefore it follows from \eqref{13} and \eqref{13x} that if $r(\mu_0)\neq1$ then
        \[\begin{array}{ll}
            \mathscr{D}(s;\mu) &= D_m\circ\dots\circ D_1(s;\mu)-D_{m+1}^{-1}\circ\dots\circ D_n^{-1}(s;\mu) \vspace{0.2cm} \\
            &= A_{1,m}s^{\Lambda_{0,m}}-A_{m+1,n}^*s^{\Lambda_{m,n}^{-1}}+\mathcal{F}^\infty_\ell(\mu_0),
        \end{array}\]
    for any given
        \[\ell\in(\max\{\Lambda_{0,m}^0,(\Lambda_{m,n}^0)^{-1}\},\min\{\Lambda_{0,m}^0+1,(\Lambda_{m,n}^0)^{-1}+1\}).\]   
    Suppose now that $r(\mu_0)=1$ and let $\alpha=\Lambda_{m,n}^{-1}-\Lambda_{0,m}$. In this case we have,
    \begin{equation}\label{14}
    \begin{array}{ll}
        \mathscr{D}(s;\mu) &=D_m\circ\dots\circ D_1(s;\mu)-D_{m+1}^{-1}\circ\dots\circ D_n^{-1}(s;\mu) \vspace{0.2cm} \\
        
        &=s^{\Lambda_{0,m}}\bigl(A_{1,m}+B_{1,m}s+\mathcal{F}^\infty_{\ell_1}(\mu_0)\bigr)-s^{\Lambda_{m,n}^{-1}}\bigl(A_{m+1,n}^*+B_{m+1,n}^*s+\mathcal{F}^\infty_{\eta_n}(\mu_0)\bigr) \vspace{0.2cm} \\

        &=s^{\Lambda_{m,n}^{-1}}\bigl[s^{-\alpha}\bigl(A_{1,m}+B_{1,m}s+\mathcal{F}^\infty_{\ell_1}(\mu_0)\bigr)-\bigr(A_{m+1,n}^*+B_{m+1,n}^*s+\mathcal{F}^\infty_{\eta_n}(\mu_0)\bigr)\bigr].
    \end{array}
    \end{equation}
    Consider 
        \[U(s;\mu)=1+\frac{B_{1,m}}{A_{1,m}}s+\mathcal{F}^\infty_{\ell_1}(\mu_0),\]
    and observe that from the Generalized Binomial Theorem~\ref{GBT} we have
        \[U(s;\mu)^{-1}=1-\frac{B_{1,m}}{A_{1,m}}s+\mathcal{F}^\infty_{\ell_1}(\mu_0).\]
    Hence it follows from \eqref{14} that 
    \[\begin{array}{ll}
        \mathscr{D}(s;\mu) &= s^{\Lambda_{m,n}^{-1}}\bigl[s^{-\alpha}A_{1,m}U(s;\mu)-\bigl(A_{m+1,n}^*+B_{m+1,n}^*s+\mathcal{F}^\infty_{\eta_n}(\mu_0)\bigr)\bigr] \vspace{0.2cm} \\

        &= s^{\Lambda_{m,n}^{-1}}U(s;\mu)\bigl[s^{-\alpha}A_{1,m}-U(s;\mu)^{-1}\bigl(A_{m+1,n}^*+B_{m+1,n}^*s+\mathcal{F}^\infty_{\eta_n}(\mu_0)\bigr)\bigr] \vspace{0.2cm} \\

        &\displaystyle= s^{\Lambda_{m,n}^{-1}}U(s;\mu)\biggl[\bigl(1+\alpha\omega(s;\alpha)\bigr)A_{1,m} \bigr.\biggr. \vspace{0.2cm} \\
        
        &\displaystyle \biggl.\bigl. \qquad\qquad\qquad\qquad\qquad-\biggl(1-\frac{B_{1,m}}{A_{1,m}}s+\mathcal{F}^\infty_{\ell_1}\biggr)\bigl(A_{m+1,n}^*+B_{m+1,n}^*s+\mathcal{F}^\infty_{\eta_n}(\mu_0)\bigr)\biggr] \vspace{0.2cm} \\

        &\displaystyle= s^{\Lambda_{m,n}^{-1}}U(s;\mu)\bigl[\Psi_1\omega(s;\alpha)+\Psi_2+\Psi_3s+\mathcal{F}^\infty_{\ell'}(\mu_0)\bigr],
    \end{array}\]
    where $\Psi_1=\alpha A_{1,m}$, $\Psi_2=A_{1,m}-A_{m+1,n}^*$,
        \[\begin{array}{ll}
            \Psi_3 &\displaystyle= A_{m+1,n}^*\frac{B_{1,m}}{A_{1,m}}-B_{m+1,n}^* \vspace{0.2cm} \\
            
            &\displaystyle= A_{m+1,n}^*\left(\Lambda_{1,m}\frac{\Delta_{10}^1}{\Delta_{00}^1}+\Lambda_{m,n}^{-1}\frac{\Delta_{01}^n}{(\Delta_{00}^n)^2}\right) \vspace{0.2cm} \\
            
            &\displaystyle= A_{m+1,n}^*\bigl(\Lambda_{0,m}S_1^1-\Lambda_{m,n}^{-1}S_2^n),
        \end{array}\]
    with the second equality following from \eqref{15} and 
        \[\ell'\in(1,\min\{\ell_1,\ell_n,2\})=(1,\min\{\lambda_1^0,(\lambda_n^0)^{-1},2\}).\]
    Finally, we now observe that $U(s;\mu)=1+\Lambda_{0,m}S_1^1s+\mathcal{F}^\infty_{\ell'}(\mu_0)$.
\end{proof}

\begin{obs}
    Similarly to Remark~\ref{R2}, we observe that except by a change of indexation, Proposition~\ref{Propo:displacement+-} is also equivalent to Propositions~\ref{Propo:Return1-+} and~\ref{Propo:Return1+-}.
\end{obs}

\section{Proof of the main theorems}\label{Sec:Proofs}
\begin{proof}[Proof of Theorem~\ref{Teo:Return0Cycl012}]
The expression of the return map \eqref{eq:ReturnmapL0} follows directly from Corollary~\ref{Coro0}. As for the assertions concerning the cyclicity of $\Gamma$, we need to consider the displacement map given by:
\begin{equation*}
\mathscr{D}(s;\mu)=\mathscr{R}(s;\mu)-s=s^r\left(A_{1,n}-s^{1-r}+\mathcal{F}_\ell^\infty(\mu_0)\right),
\end{equation*}
It is known that no bifurcation of limit cycles occur at $\mu_0$ near a persistent polycycle $\Gamma$ with graphic number $r(\mu_0)\neq 1$ (See \cite[Remark 2.12]{QMV}). This is precisely statement (a). To prove (b), we write (recall Definition~\ref{Def8}) the displacemet map as
\begin{equation}\label{eq:proof1.1}
\mathscr{D}(s;\mu)=s^{r}\left(A_{1,n}-1+(1-r)\omega(s;r-1)+\mathcal{F}_\ell^\infty(\mu_0)\right).
\end{equation}
Now, since $\mathscr{R}(\cdot,\mu_0)\not\equiv Id$, given $\varepsilon>0$ there exists $s_1\in(0,\varepsilon)$ such that $\mathscr{D}(s_1,\mu_0)\neq 0$. Without loss of generality, assume that $\mathscr{D}(s_1,\mu_0)>0$, since the other case is analogous. Thus, there exists a neighborhood $U$ of $\mu_0$ such that $\mathscr{D}(s_1,\mu)>0$ for $\mu\in U$. Since $r-1$ changes signs at $\mu_0$ we can take $\mu_1\in U$ such that $r(\mu_1)>1$. Now, using Lemmas \ref{LemaPropFlat} and \ref{Lema:omega} we have that
\begin{equation}\label{eq:proof:1.2}
Z_1(s;\mu):=\frac{s^{-r}\mathscr{D}(s;\mu)}{\omega(s;r-1)}=\frac{A_{1,n}-1}{\omega(s;r-1)}+(1-r)+\mathcal{F}_{\ell-\delta}^\infty(\mu_0),
\end{equation}
and
\begin{equation*}
\lim\limits_{s\to 0^+}Z_1(s;\mu)=(A_{1,n}-1)\max\{1-r,0\}+(1-r)=(1-r)\beta,
\end{equation*}
where 
    \[\beta=\left\{\begin{array}{ll}
                    1, & \text{if } 1-r\leqslant0, \vspace{0.2cm} \\
                    A_{1,n}, & \text{if } 1-r>0.
    \end{array}\right.\]
In either case, we have $\beta>0$. Since $r(\mu_1)<1$, we have that $\lim\limits_{s\to 0^+}Z_1(s;\mu_1)<0$ and thus, there exists $s_2\in (0,s_1)$ such that $Z_1(s_2,\mu_1)<0$. Therefore $\mathscr{D}(s_2,\mu_1)\mathscr{D}(s_1,\mu_1)<0$ and by continuity, there is at least one $s^\ast\in (s_2,s_1)\subset(0,\varepsilon)$ such that $\mathscr{D}(s^\ast,\mu_1)=0$. Since $X_{\mu_1}$ is analytic, we have that $s^*$ is an isolated solution and thus ${\rm Cycl}(\Gamma,\mu_0)\geqslant1$.

Now, we turn to prove statement (c). For $r(\mu_0)\neq 1$ this statement follows from (a). For $r(\mu_0)=1$, the upper bound on the cyclicity is obtained by applying the derivation-division algorithm. Using Lemma~\ref{Lema:omega} we have that
\begin{equation}\label{eq:proof1.4}
\partial_sZ_1(s;\mu)=\frac{(A_{1,n}-1)s^{-r}}{\omega(s;r-1)^2}+\mathcal{F}_{\ell-\delta-1}^\infty (\mu_0).
\end{equation}
Since $s^{-\alpha}\in\mathcal{F}_{-\delta}^{\infty}(\{\alpha=0\})$, for any given $\delta\in(0,\ell/4)$ we have that
\[\lim\limits_{(s,\mu)\to (0^+,\mu_0)}s^r\omega(s;r-1)^2\partial_sZ_1(s;\mu)=\lim\limits_{(s,\mu)\to (0^+,\mu_0)}(A_{1,n}-1)+\mathcal{F}_{\ell-4\delta}^\infty (\mu_0)=A_{1,n}(\mu_0)-1.\]
Under the hypothesis of (c), the above limit is not zero, which implies by Rolle's Theorem that there is a small neighborhood of $\mu_0$ and $\varepsilon>0$ such that $Z_1(\cdot;\mu)$ and thus $\mathscr{D}(\cdot;\mu)$ has at most one zero $s^\ast\in(0,\varepsilon)$. Hence, ${\rm Cycl}(\Gamma,\mu_0)\leqslant 1$.

Finally, we assume the hypothesis of (d). Since $\mathscr{R}(\cdot;\mu_0)\not\equiv Id$, given $\varepsilon>0$, there exists $s_1\in (0,\varepsilon)$ such that $\mathscr{D}(s_1;\mu_0)\neq 0$. Again, we assume without loss of generality that that $\mathscr{D}(s_1,\mu_0)<0$ which implies that there exists  a neighborhood $U$ of $\mu_0$ such that $\mathscr{D}(s_1,\mu)<0$ for $\mu\in U$. Since $r-1$ and $A_{1,n}-1$ are independent at $\mu_0$, we can take $\mu_1\in U$ such that $r(\mu_1)=1$ and $A_{1,n}(\mu_1)>1$. Then, by \eqref{eq:proof1.1}
\[\lim_{s\to 0^+}s^{-1}\mathscr{D}(s;\mu_1)=A_{1,n}(\mu_1)-1>0.\]
Therefore, there exists $s_2\in(0,s_1)$ such that $\mathscr{D}(s_2,\mu_1)>0$, which implies that there is a neighborhood $U_1\subset U$ of $\mu_1$ such that $\mathscr{D}(s_2,\mu)>0$ for $\mu\in U_1$. Now, using the independence of $r-1$ and $A_{1,n}-1$ at $\mu_0$, we take $\mu_2\in U_1$ such that $r(\mu_2)>1$. From \eqref{eq:proof1.4}, we have that $\lim\limits_{s\to 0^+}Z_1(s;\mu_2)<0$ which implies the existence of $s_3\in(0,s_2)$ such that $\mathscr{D}(s_3;\mu_2)<0$. Now, since 
    \[\mathscr{D}(s_3;\mu_2)<0,\;\mathscr{D}(s_2;\mu_2)>0,\;\mathscr{D}(s_1;\mu_2)<0,\]
we have by continuity that ${\rm Cycl}(\Gamma,\mu_0)\geqslant 2$.
\end{proof}

\medskip

\begin{proof}[Proof of Theorem~\ref{Teo:Return1-+Cycl23}]
It follows from Proposition \ref{Propo:Return1-+} that the return map is given by equation~\eqref{eq:Return1-+}. Thus, we turn to the proof of the statements concerning the cyclicity of $\Gamma$. For this purpose, we consider the displacement map $\mathscr{D}(s;\mu)=\mathscr{R}(s;\mu)-s$, which under the current hypothesis is written as
\begin{equation}\label{eq:proof:2.0}
\mathscr{D}(s;\mu)=s^{r}\big(A_{1,n}-s^{1-r}+\mathcal{A}s^{\Lambda_{0,m}}+\mathcal{F}^\infty_{\ell}(\mu_0)\bigr).
\end{equation}

To prove (a) we apply the derivation-division algorithm to the function $Z_1(s;\mu)$ defined in \eqref{eq:proof:1.2}, which under the current hypothesis writes as follows.
\begin{equation*}
Z_1(s;\mu)=\frac{A_{1,n}-1}{\omega(s;r-1)}+(1-r)+\frac{\mathcal{A}s^{\Lambda_{0,m}}}{\omega(s;r-1)}+\mathcal{F}_{\ell}^{\infty}(\mu_0).
\end{equation*}
We assume that the hypothesis of items (a) and (c) in Theorem \ref{Teo:Return0Cycl012} do not hold, i.e. $r(\mu_0)=1$ and $A_{1,n}(\mu_0)=1$, otherwise we would have ${\rm Cycl}(\Gamma,\mu_0)<2$ immediately. We have that
\begin{equation*}
\partial_sZ_1(s;\mu)=\frac{(A_{1,n}-1)s^{-r}}{\omega(s;r-1)^2}+\frac{\Lambda_{0,m}\mathcal{A}s^{\Lambda_{0,m}-1}}{\omega(s;r-1)}+\frac{\mathcal{A}s^{\Lambda_{0,m}-r}}{\omega^2(s;r-1)}+\mathcal{F}_{\ell-\delta-1}^{\infty}(\mu_0),
\end{equation*}
which yield
\begin{eqnarray*}
\Theta_1(s;\mu)&:=&s^r\omega^2(s;r-1)\partial_sZ_1(s;\mu)\nonumber\\
&=&(A_{1,n}-1)+\Lambda_{0,m}\mathcal{A}s^{\Lambda_{0,m}+r-1}\omega(s;r-1)+\mathcal{A}s^{\Lambda_{0,m}}+\mathcal{F}_{\ell-4\delta}^{\infty}(\mu_0),
\end{eqnarray*}
for any $\delta\in (0,\ell/4)$. The derivative with respect to $s$ is given by
\begin{equation*}
\partial_s\Theta_1(s;\mu)=(r-1+\Lambda_{0,m})\Lambda_{0,m}\mathcal{A}s^{r+\Lambda_{0,m}-2}\omega(s;r-1)+\mathcal{F}_{\ell-4\delta-1}^{\infty}(\mu_0).
\end{equation*}
and finally,
\begin{equation*}
Z_2(s;\mu):=\frac{s^{2-\Lambda_{0,m}-r}\partial_s\Theta_1(s;\mu)}{\omega(s;r-1)}=(r-1+\Lambda_{0,m})\Lambda_{0,m}\mathcal{A}+\mathcal{F}_{\ell+1-6\delta}^{\infty}(\mu_0).
\end{equation*}
Taking $\delta\in (0,\min\{\ell/4,(\ell+1)/6\})$, we have that
\[\lim\limits_{(s,\mu)\to(0^+,\mu_0)}Z_2(s;\mu)=(\Lambda_{0,m}^0)^2\mathcal{A}(\mu_0)\neq 0,\]
and by Rolle's theorem, there is a small neighborhood $U$ of $\mu_0$ and $\varepsilon>0$ such that $\Theta_1(\cdot;\mu)$ has at most one zero in $(0,\varepsilon)$ which implies that $Z_1(s;\mu)$ and thus $\mathscr{D}(s;\mu)$ have at most two zeros in the interval $(0,\varepsilon)$. Thus, ${\rm Cycl}(\Gamma,\mu_0)\leqslant 2$.

To prove the assertion in item (b), we follow steps analogous to those in the proofs of items (b) and (d) in Theorem \ref{Teo:Return0Cycl012}: Since $\mathscr{R}(\cdot;\mu_0)\not\equiv Id$, given $\varepsilon>0$, there exists a $s_1\in (0,\varepsilon)$ such that $\mathscr{D}(s_1,\mu_0)\neq 0$, without loss of generality, we assume that $\mathscr{D}(s_1,\mu_0)>0$. By continuity, there exists $U\ni \mu_0$ such that $\mathscr{D}(s_1;\mu)>0$ for $\mu\in U$. Then, by the independence of $r-1$, $A_{1,n}-1$ and $\mathcal{A}$ at $\mu_0$, we take $\mu_1\in U$ for which $r(\mu_1)=1$, $A_{1,n}(\mu_1)=1$ and $\mathcal{A}(\mu_1)<0$. Then from~\eqref{eq:proof:2.0} we have
\[\lim\limits_{s\to 0^+}s^{-1-\Lambda_{0,m}}\mathscr{D}(s;\mu_1)=\mathcal{A}(\mu_1)<0,\]
which implies that there exists $s_2\in (0,s_1)$ such that $\mathscr{D}(s_2;\mu_1)<0$. Hence, by continuity, there exists a neighborhood $U_1\subset U$ of $\mu_1$ such that $\mathscr{D}(s_2;\cdot)\vert_{U_1}<0$. Again, we take $\mu_2\in U_1$ with $r(\mu_2)=1$ and $A_{1,n}(\mu_2)>1$. By \eqref{eq:proof:2.0} we have
\[\lim\limits_{s\to 0^+}s^{-1}\mathscr{D}(s;\mu_2)=A_{1,n}(\mu_2)-1>0.\]
Thus, there exists $s_3\in (0,s_2)$ such that $\mathscr{D}(s_3;\mu_2)>0$ which in turn implies that there exists a neighborhood $U_2$, with $\mu_2\in U_2\subset U_1$ such that $\mathscr{D}(s_3;\cdot)\vert_{U_2}>0$. Finally, taking $\mu_3\in U_2$ for which $r(\mu_3)>1$, we obtain that
\[\lim\limits_{s\to 0^+}Z_1(s;\mu_3)=1-r(\mu_3)<0.\]
Hence, we obtain $s_4\in (0,s_3)$ such that $\mathscr{D}(s_4;\mu_3)<0$, $\mathscr{D}(s_3;\mu_3)>0$, $\mathscr{D}(s_2;\mu_3)<0$ and $\mathscr{D}(s_1;\mu_3)>0$ and therefore we conclude that ${\rm Cycl}(\Gamma,\mu_0)\geqslant 3$.
\end{proof}

\section{The equivalence of the displacement map}\label{Sec:displacement2}

As anticipated in Section~\ref{Sec:displacement1}, in this paper we focused on the first return map. However, sometimes it may be convenient to work with the displacement map~\eqref{23} instead. Therefore in this section we observe that there is also a similar version of Theorems~\ref{Teo:Return0Cycl012} and~\ref{Teo:Return1-+Cycl23} for the displacement map.

\begin{main}\label{Teo:displacement}
 Let $\{X_\mu\}_{\mu\in\Lambda}$ be a smooth family of planar analytic vector fields having a persistent polycycle $\Gamma$ with hyperbolic saddles $p_1,\dots,p_m,p_{m+1},\dots,p_n$. Let $\mu_0\in\Lambda$ be such that $\lambda_i(\mu_0)<1$ for $i\in\{1,\dots,m\}$ and $\lambda_i(\mu_0)>1$ for $i\in\{m+1,\dots,n\}$. Then the first displacement map of $\Gamma$ is given by
\begin{equation*}
        \mathscr{D}(s;\mu)=\left\{\begin{array}{ll}
            A_{1,m}s^{\Lambda_{0,m}}-A_{m+1,n}^*s^{\Lambda_{m,n}^{-1}}+\mathcal{F}^\infty_{\ell}(\mu_0), & \text{if } r(\mu_0)\neq1, \vspace{0.2cm} \\
            
            s^{\Lambda_{m,n}^{-1}}U(s;\mu)                      \bigl(\Psi_1\omega(s;\alpha)+\Psi_2+\Psi_3s+\mathcal{F}^\infty_{\ell'}(\mu_0)\bigr), & \text{if } r(\mu_0)=1,
        \end{array}\right.
\end{equation*}
for any given 
   \[\ell\in(\max\{\Lambda_{0,m}^0,(\Lambda_{m,n}^0)^{-1}\},\min\{\Lambda_{0,m}^0+1,(\Lambda_{m,n}^0)^{-1}+1\}), \quad \ell'\in(1,\min\{\lambda_1^0,(\lambda_n^0)^{-1},2\}),\]
where 
\begin{equation*}
    \Psi_1=\alpha A_{1,m}, \quad \Psi_2=A_{1,m}-A_{m+1,n}^*, \quad \Psi_3=A_{m+1,n}^*\bigl(\Lambda_{0,m}S_1^1-\Lambda_{m,n}^{-1}S_2^n),
\end{equation*}
$\alpha=\Lambda_{m,n}^{-1}-\Lambda_{0,m}$ and $U(s;\mu)=1+\Lambda_{0,m}S_1^1s+\mathcal{F}^\infty_{\ell'}(\mu_0)$.
Moreover, the following holds:
\begin{itemize}
    \item [(a)] ${\rm Cycl}(\Gamma,\mu_0)=0$, if $\Psi_1\neq0$;
    \item [(b)] ${\rm Cycl}(\Gamma,\mu_0)\geqslant 1$, if $\Psi_1=0$, $\Psi_1$ changes signs at $\mu_0$ and $\mathscr{R}(\cdot;\mu_0)\not\equiv Id$;
    \item [(c)] ${\rm Cycl}(\Gamma,\mu_0)\leqslant 1$, if $\Psi_2\neq0$;
    \item [(d)] ${\rm Cycl}(\Gamma,\mu_0)\geqslant 2$, if  $\Psi_1=\Psi_2=0$, $\Psi_1,\;\Psi_2$ are independent at $\mu_0$ and $\mathscr{R}(\cdot;\mu_0)\not\equiv Id$;  
    \item [(e)] ${\rm Cycl}(\Gamma,\mu_0)\leqslant 2$, if $\Psi_3\neq0$;
    \item [(f)] ${\rm Cycl}(\Gamma,\mu_0)\geqslant 3$ if $\Psi_1=\Psi_2=\Psi_3=0$, $\Psi_1,\;\Psi_2,\;\Psi_3$ are independent at $\mu_0$ and $\mathscr{R}(\cdot;\mu_0)\not\equiv Id$.
\end{itemize}
\end{main}

\begin{proof}
    The expression of $\mathscr{D}$ follows from Proposition~\ref{Propo:displacement+-}. The proof of the statements about the cyclicity follows similarly to the proof of Theorems~\ref{Teo:Return0Cycl012} and~\ref{Teo:Return1-+Cycl23} (also also similarly to the proof of~\cite[Theorem~$A$]{MarVilKolmogorov}).
\end{proof}

We finish this section by observing that the first return and the displacement maps also shares another type of similarity. More precisely, from Theorems~\ref{Teo:Return0Cycl012} and~\ref{Teo:Return1-+Cycl23} it follows that the cyclicity of $\Gamma$ is governed by the zeros of the functions
    \[\Phi_1:=r(\mu)-1, \quad \Phi_2:=A_{1,n}-1, \quad \Phi_3:=\mathcal{A},\]
where we recall that $\mathcal{A}=\Lambda_{m,n}A_{1,m}A_{1,n}(S_1^{m+1}-S_2^m)$. Therefore if we let $\Psi_1$, $\Psi_2$ and $\Psi_3$ be given as in Theorem~\ref{Teo:displacement}, then one can apply the formulas \eqref{eq:formulasAlambda} to verify that
    \[V(\Phi_1)=V(\Psi_1), \quad V(\Phi_1,\Phi_2)=V(\Psi_1,\Psi_2), \quad V(\Phi_1,\Phi_2,\Phi_3)=V(\Psi_1,\Psi_2,\Psi_3),\]
where we recall that $V(f_1,\dots,f_k)$ denotes the variety defined by $f_1,\dots,f_k$ (see Definition~\ref{Def1}).

\section{An application in Game Theory}\label{Sec:Application}

The notion of Evolutionary Stable States (ESS) was first introduced in the paper \cite{SmithPrice73} by Smith and Price, in which they applied concepts of Game Theory into Biology. Roughly speaking, given a game with two or more players (modeling a conflict between species, for instance), an ESS is an strategy such that if most of the players follow it, then no other strategy would provide the other players higher advantages, that is, the best course of action for the other players is to also follow the ESS.

In 1978, Taylor and Jonker \cite{TaylorJonker78} approached the study of ESS in the scope of Ordinary Differential Equations. One of their significant contributions was the modeling of a multiple-player game by a system of differential equations. In particular, a game with two players can be modeled by a planar polynomial vector field. One of such models is given by the following polynomial system. 

\begin{equation}\label{eq:gamemodel}
\begin{array}{l}
\dot x= x(x-1)f(x,y),\\
\dot y= y(y-1)g(x,y).
\end{array}
\end{equation}

In the context given by the model \eqref{eq:gamemodel}, the limit cycles have an important significance. Hofbauer et al. \cite{Hofbauer} proved that every ESS is an assymptotically stable singularity, while the converse does not hold. They also observed that there is no special distinction between ESS and assymptotically stability. Hence, one can study assymptotical stability rather than ESS. In this scenario, a stable limit cycle can be interpreted as an \emph{oscillating stable strategy}. In this scope, the model \eqref{eq:gamemodel} has been recently studied in several papers (for instance, \cite{BastosBuzziSantana24,BastosBuzziSantana24JDE,GasullFofoSantana25}).

In the present work, we consider system \eqref{eq:gamemodel} in the case which the boundary of the unit square is a hyperbolic polycycle. More precisely, we work with the family $X_{\mu}$ of vector fields associated to the following systems.

\begin{equation}\label{eq:gamepoly}
\begin{array}{l}
\dot x=x\, (x-1)\biggl(-1-\left(\lambda_{3}-1\right) x +y -\left(\lambda_{1}-\lambda_{3}\right) x y + \lambda_{1} y^{2}\biggr),\\
\dot y=y\,(y-1)\biggl(\lambda_{2}-\left(\lambda_{2}+\mu_{1}\right) x -\left(\lambda_{2}-1\right) y +\left(\mu_{1}-1\right) x^{2}+\left(\lambda_{2}-\lambda_{4}\right) x y\biggr).
\end{array}
\end{equation}

\begin{teo}
There exist parameter values $\mu_{0}\in\Lambda$, such that $X_{\mu}$ has two limit cycles bifurcating from the polycycle at the boundary of the unit square for $\mu\approx \mu_0$.
\end{teo}

\begin{proof}
Consider the family $\{X_\mu\}_{\mu\in\Lambda}$ associated to \eqref{eq:gamepoly}, with $\mu=(\lambda_1,\lambda_2,\lambda_3,\lambda_4,\mu_1)$ and $\Lambda=\{\lambda_i>0: i\in\{1,\dots,4\}\}$. Denote by $\Gamma$ the boundary of $[1,0]^2$. We have that $\Gamma$ is a persistent polycycle for the family $\{X_\mu\}_{\mu\in\Lambda}$. Indeed, we have that the lines $x=0$, $x=1$, $y=0$ and $y=1$ are invariant through $X_{\mu}$. Moreover, the points $p_1=(0,1)$, $p_2=(0,0)$, $p_3=(1,0)$ and $p_4=(1,1)$ are hyperbolic saddles since the Jacobian matrix of $X_{\mu}$ evaluated at $p_i$ is given by
$$JX_{\mu}(p_1)=\left(\begin{array}{cc}
    -\lambda_1 & 0 \\
    0 & 1
\end{array}\right),\quad JX_{\mu}(p_2)=\left(\begin{array}{cc}
    1 & 0 \\
    0 & -\lambda_2
\end{array}\right),$$
$$ JX_{\mu}(p_3)=\left(\begin{array}{cc}
    -\lambda_3 & 0 \\
    0 & 1
\end{array}\right),\quad JX_{\mu}(p_4)=\left(\begin{array}{cc}
    1 & 0 \\
    0 & -\lambda_4
\end{array}\right).$$
Since the quantities 
$$\frac{\lambda_1}{\lambda_1-1},\;\frac{\lambda_2}{\lambda_2-1},\;\frac{1}{1-\lambda_3},\;\frac{1}{1-\lambda_4},$$
do not lie in the interval $(0,1)$ for $\mu\in\Lambda$, there are no singularities in $\Gamma$ besides $p_i$. Thus, $\Gamma$ is a persistent polycycle for $\mu\in\Lambda$. Notice that $\Gamma$ is oriented counterclockwise.

Since $\Gamma$ is a square, simple translations and rotations suffice to put system \eqref{eq:gamemodel} into the standard form \eqref{eq:X1}. Thus, one can readily apply the formulas given in \cite[Theorem A]{MarVilDulacCoef} to compute the coefficients $\Delta_{00}^i$, $\Delta_{10}^i$ and $\Delta_{01}^i$ of each Dulac map $D_i(s;\mu)$. We need to compute the functions $r(\mu)$, $A_{1,4}(\mu)$ of the return map, to study the cyclicity of $\Gamma$. We have that
\begin{equation*}
\begin{aligned}
r(\mu)&=\lambda_1\lambda_2\lambda_3\lambda_4,\\ 
A_{1,4}(\mu)&=\exp\biggl(\frac{1}{(\lambda_{1}-1) (\lambda_{2}-1) (\lambda_{3}-1) (\lambda_{4}-1) \lambda_{1}}\bigg(\ln (\lambda_{1}+1) (r(\mu) -1) (\lambda_{4}-1) (\lambda_{3}-1)\cdot\\
&\;(\lambda_{2}-1) (1-\lambda_{1}^{2} \lambda_{4}+(\mu_{1}+\lambda_{4}-2) \lambda_{1})+\lambda_{1} (\ln (2) (\lambda_{4}-1) (\lambda_{3}-1) (\lambda_{2}-1) (1-\mu_{1}) r(\mu)\\
&\;\;-\lambda_{2} \lambda_{3} (\lambda_{4}-1) (\lambda_{3}-1) \lambda_{4} (\lambda_{2}-1) (1-\lambda_{1}^{2} \lambda_{4}+(\mu_{1}+\lambda_{4}-2) \lambda_{1}) \ln (\lambda_{1})+(\lambda_{2} \lambda_{3}^{2}-1\\
&\;\;\;+(\mu_{1}-\lambda_{2}) \lambda_{3}) (\lambda_{4}-1) \lambda_{4} (\lambda_{1}-1) (\lambda_{2}-1) \ln (\lambda_{3})-(\lambda_{3}-1) ((\lambda_{1}-1) (\lambda_{2}-1) (\lambda_{1} \lambda_{4}\\
&\;\;\;\;-(\lambda_{4}-1) (\lambda_{4} \lambda_{3}+1)) \ln (\lambda_{4})-(\lambda_{4}-1) ((\lambda_{2}-1) (-1+\mu_{1}) \ln (2)+\lambda_{3} \lambda_{4} \ln (\lambda_{2})\cdot\\
&\;\;\;\;\;(\lambda_{1}-1) (\lambda_{1} \lambda_{2}^{2}+\lambda_{2}-1))))\bigg)
\biggr).
\end{aligned}
\end{equation*}


For $\mu_0=(\frac{8}{27},\frac{3}{2},\frac{3}{2},\frac{3}{2},\frac{2}{5},\frac{1625}{162})$, we have $r(\mu_0)=A_{1,4}(\mu_0)=1$ and 
$${\rm rank}\left(\dfrac{\partial(r-1,A_{1,4}-1)}{\partial\mu}\right)_{\mu=\mu_0}=2,$$
which implies that $r-1$ and $A_{1,4}-1$ are independent at $\mu_0$. To check if ${\rm Cycl}(\Gamma,\mu_0)\geqslant 2$, by Theorem \ref{Teo:Return0Cycl012}, we need to verify if the return map $\mathscr{R}(s;\mu_0)$ is not identically the identity map. In order to do so, we compute the expression $B(\mu_0)=S_{1}^{2}-S_{2}^{1}$. Notice that for $\mu=\mu_0$, we are under the hypothesis of Theorem \ref{Teo:Return1-+Cycl23} and that $B(\mu_0)$ is a factor of $\mathcal{A}(\mu_0)$ such that $B(\mu_0)\neq 0$ implies $\mathcal{A}(\mu_0)\neq 0$. The full expression of $B(\mu_0)$ is too cumbersome, so we omit it from the text. Its numerical value up to 12 decimal places is $B(\mu_0)\approx 6.20031365865$. By Theorem \ref{Teo:Return1-+Cycl23}, we have that ${\rm Cycl}(\Gamma,\mu_0)=2$.
\end{proof}

\section*{Acknowledgments}
\noindent The first author is supported by Funda\c c\~ao de Amparo \`a Pesquisa do Estado de S\~ao Paulo (grant number 	
2024/06926-7). The second is supported by Funda\c c\~ao de Amparo \`a Pesquisa do Estado de S\~ao Paulo (grant number 	
2021/01799-9)

\section*{Conflict of interest}

\noindent The authors declare that they have no conflict of interest.

\appendix

\section{Generalized Binomial Theorem}\label{App:GBT}

Given $\alpha\in\mathbb{C}$ and $k\in\mathbb{Z}_{\geqslant0}$, the \emph{generalized binomial coefficient} is given by
    \begin{equation}\label{0}         
        \binom{\alpha}{k}=\frac{\alpha(\alpha-1)\dots(\alpha-k+1)}{k!},
    \end{equation}
with the convention $\binom{\alpha}{0}=1$. Observe that if $\alpha\in\mathbb{Z}_{\geqslant k}$, then \eqref{0} reduces to the usual binomial coefficient.

\begin{teo}[Generalized Binomial Theorem]\label{GBT}
    Let $x$, $y$, $\alpha\in\mathbb{C}$ such that $|x|>|y|$. Then 
    \[(x+y)^\alpha=\sum_{k=0}^{\infty}\binom{\alpha}{k}x^{\alpha-k}y^k.\]
\end{teo}

\begin{proof}
    Since $x\neq0$ it follows that $t=y/x\in\mathbb{C}$ is well defined and thus we can consider the holomorphic function $f\colon\mathbb{C}\to\mathbb{C}$ given by,
        \[f(t)=(1+t)^\alpha.\]
    From $f^{(k)}(t)=\alpha(\alpha-1)\dots(\alpha-k+1)(1+t)^{\alpha-k}$ it follows that expanding $f$ in Maclaurin's series we get
        \[(1+t)^\alpha=\sum_{k=0}^{\infty}\frac{f^{(k)}(0)}{k!}t^k=\sum_{k=0}^{\infty}\frac{\alpha(\alpha-1)\dots(\alpha-k+1)}{k!}t^k=\sum_{k=0}^{\infty}\binom{\alpha}{k}t^k,\]
    with the series converging provided $|t|<1$. The theorem now follows by replacing $t=y/x$ and multiplying the equation by $x^\alpha$.
\end{proof}

\section{Coefficient expressions for the Dulac map}\label{App:Dulac}
In this section, we present the explicit expressions for the coefficients $\Delta_{00}, \Delta_{10}$ and $\Delta_{01}$ obtained by Marín and Villadelprat in \cite[Theorem A]{MarVilDulacCoef}. Considering the vector field \eqref{eq:X1}, we define the following functions:

\begin{align}
&L_1(u):=\exp\int_{0}^{u}\biggl(\frac{P(0,y;\mu)}{Q(0,y;\mu)}+\frac{1}{\lambda}\biggr)\frac{dy}{y},&\quad &L_2(u):=\exp\int_{0}^{u}\biggl(\frac{Q(x,0;\mu)}{P(x,0;\mu)}+\lambda\biggr)\frac{dx}{x},\nonumber\\
&M_1(u):=L_1(u)\partial_x\left(\frac{P}{Q}\right)(0,u),&\quad & M_2(u):=L_2(u)\partial_y\left(\frac{Q}{P}\right)(u,0),    
\end{align}
Let $\sigma_{ijk}$ denote the $k$th derivative at $s=0$ of the $j$th component of the transverse section  $\sigma_i=(\sigma_{i,1},\sigma_{i,2})$, more precisely,
$$\sigma_{ijk}=\partial_s^k\sigma_{i,j}(0;\mu).$$
Now, we define the following quantities:
\begin{align}
&S_1:=\frac{\sigma_{112}}{2\sigma_{111}}-\frac{\sigma_{121}}{\sigma_{120}}\left(\frac{P}{Q}\right)(0,\sigma_{120})-\frac{\sigma_{111}}{L_1(\sigma_{120})}\hat{M}_1(1/\lambda,\sigma_{120}),\nonumber\\
&S_2:=\frac{\sigma_{222}}{2\sigma_{221}}-\frac{\sigma_{211}}{\sigma_{210}}\left(\frac{Q}{P}\right)(\sigma_{210},0)-\frac{\sigma_{221}}{L_2(\sigma_{210})}\hat{M}_2(\lambda,\sigma_{210}),
\end{align}
where $\hat{M}_i$ denotes a sort of incomplete Melin transform. We refer the reader to~\cite[Appendix~B]{MarVilDulacCoef} for a detailed study. For our purposes, the following result suffices to perform accurate computations.

\begin{propo}[{\cite[[Theorem B.1]{MarVilDulacCoef}}]
Consider an open interval $I\subset\mathbb{R}$ containing $x=0$ and an open subset $U\subset\mathbb{R}^N$.
\begin{itemize}
    \item[(a)] Given $f(x;\nu)\in\mathscr{C}^\infty(I\times U)$, there exists a unique $\hat{f}(\alpha,x;\nu)\in\mathscr{C}^\infty((\mathbb{R}\setminus\mathbb{Z}_{\geqslant 0})\times I\times U)$ such that $$x\partial_x\hat{f}(\alpha,x;\nu)-\alpha\hat{f}(\alpha,x;\nu)=f(x;\nu);$$
\item[(b)] If $x\in I\setminus\{0\}$, then $\partial_x(\hat{f}(\alpha,x;\nu)|x|^{-\alpha})=f(x;\nu)\frac{|x|^{-\alpha}}{x}$ and, taking any $k\in\mathbb{Z}_{\geqslant 0}$, with $k>\alpha$,
 $$\hat{f}(\alpha,x;\nu)=\sum_{i=0}^{k-1}\frac{\partial_x^if(0;\nu)}{i! (i-\alpha)}x^i+|x|^\alpha\int_{0}^{x}\biggl(f(s;\nu)-T^{k-1}_0f(s;\nu)\biggr)|s|^{-\alpha}\frac{ds}{s},$$
 where $T^k_0f(x;\nu)=\sum_{i=0}^{k}\frac{1}{i!}\partial_x^if(0;\nu)x^i$ is the $k$th degree Taylor polynomial of $f(x;\nu)$ at $x=0$;
 \item[(c)] For each $(i_0,x_0;\nu_0)\in\mathbb{Z}\times I\times U$ the function $(\alpha,x;\nu)\mapsto (i_0-\alpha)\hat{f}(\alpha,x;\nu)$ extends $\mathscr{C}^\infty$ at $(i_0,x_0;\nu_0)$ and, moreover, it tends to $\frac{1}{i_0!}\partial_x^{i_0}f(0;\nu_0)x_0^{i_0}$ as $(\alpha,x;\nu)\to (i_0,x_0;\nu_0)$;
 \item[(d)] If $f(x;\nu)$ is analytic on $I\times U$, then $\hat{f}(\alpha,x;\nu)$ is analytic on $(\mathbb{R}\setminus\mathbb{Z}_{\geqslant 0})\times I\times U$. Finally, for each $(\alpha_0,x_0;\nu_0)\in\mathbb{Z}_{\geqslant 0}\times I \times U$, the function $(\alpha,x;\nu)\mapsto (\alpha_0-\alpha)\hat{f}(\alpha,x;\nu)$ extends analytically to $(\alpha_0,x_0;\nu_0)$. 
\end{itemize}
\end{propo}

The coefficients of the Dulac map are given by the next result.

\begin{propo}[{\cite[Theorem A, item (b)]{MarVilDulacCoef}}]\label{Propo:CoefDulac}
The coefficients $\Delta_{ij}$ for $(i,j)\in\{(0,0),(1,0),(0,1)\}$ of the Dulac map are given by
\begin{equation}
\Delta_{00}=\frac{\sigma_{111}^\lambda\sigma_{120}}{L_1^\lambda(\sigma_{120})}\frac{L_2(\sigma_{210})}{\sigma_{221}\sigma_{210}^\lambda},\quad \Delta_{01}=-(\Delta_{00})^2S_2,\quad \Delta_{10}=\lambda\Delta_{00}S_1.
\end{equation}
\end{propo}

\section{An ODE model in game theory}\label{App:ODEGame}
We now briefly present the construction of model \eqref{eq:gamemodel}. Let $\Gamma_1,\Gamma_2$ be two players and $\{X_1,X_2\}$, $\{Y_1,Y_2\}$ be the respective \emph{pure strategies}. We denote by $a_{ij}^\ast\in\mathbb{R}$ the payoff of strategy $X_i$ against $Y_j$ and by $b_{ij}^\ast\in\mathbb{R}$ the payoff of strategy $Y_i$ against $X_j$. For each probabilistic vector of dimension two
$$x=(x_1,x_2)\in S^2:=\{(x_1,x_2)\in\mathbb{R}^2:x_1\geqslant 0,x_2\geqslant 0,x_1+x_2=1\},$$
we associate a \emph{mix strategy} given by $x_1 X_1+x_2 X_2$. Similarly, given $y\in S^2$, we associate the respective mix strategy $y_1Y_1+y_2Y_2$. Let
$$A^\ast=\left(\begin{array}{cc}
   a_{11}^\ast  & a_{12}^\ast \\
   a_{21}^\ast  & a_{22}^\ast 
\end{array}\right),\quad B^\ast =\left(\begin{array}{cc}
   b_{11}^\ast  & b_{12}^\ast \\
   b_{21}^\ast  & b_{22}^\ast 
\end{array}\right),$$
be the \emph{payoff matrices}. Given $x,y\in S^2$, the \emph{average payoff} of the mix strategy associated to $x$ against the mix strategy associated to $y$ is given by
$$\langle x, A^\ast y\rangle=a_{11}^\ast x_1y_1+a_{12}^\ast x_1y_2+a_{21}^\ast x_2y_1+a_{22}^\ast x_2y_2,$$
and the average payoff of the mix strategy associated to $y$ against the mix strategy associated to $x$ is given by
$$\langle y, B^\ast x\rangle=b_{11}^\ast x_1y_1+b_{12}^\ast x_1y_2+b_{21}^\ast x_2y_1+b_{22}^\ast x_2y_2.$$
The dynamics between players $\Gamma_1$ and $\Gamma_2$ is defined by the system of differential equations,
\begin{equation}\label{eq:game}
\begin{array}{l}
\dot x_1=x_1\left(\langle e_1,A^\ast y\rangle-\langle x,A^\ast y\rangle\right),\qquad \dot y_1=y_1\left(\langle e_1,B^\ast x\rangle-\langle y,B^\ast x\rangle\right),\\
\dot x_2=x_2\left(\langle e_2,A^\ast y\rangle-\langle x,A^\ast y\rangle\right),\qquad \dot y_2=y_2\left(\langle e_2,B^\ast x\rangle-\langle y,B^\ast x\rangle\right).
\end{array}
\end{equation}
Essentially, the weight $x_i$ of the pure strategy $X_i$ depends on the difference between the payoffs of the pure strategy and the mix strategy. In other words, the bigger this difference, the more superior strategy $X_i$ is. Since $x_1+x_2=y_1+y_2=1$, one can consider only the variables $x_1,y_1$ to study the dynamics of the game. Thus, \eqref{eq:game} simplifies to
\begin{equation*}
\begin{array}{l}
\dot x= x(x-1)\left(a_{22}^\ast-a_{12}^\ast+(a_{12}^\ast+a_{21}^\ast-a_{11}^\ast-a_{22}^\ast)y\right),\\
\dot y= y(y-1)\left(b_{22}^\ast-b_{12}^\ast+(b_{12}^\ast+b_{21}^\ast-b_{11}^\ast-b_{22}^\ast)x\right).
\end{array}
\end{equation*}
The search for more realistic models demanded that the payoffs depended on the weights given to strategies $X_i$ and $Y_j$ rather than being constants, i.e. $a_{ij}^\ast=a_{ij}^\ast(x,y)$ and $b_{ij}^\ast=b_{ij}^\ast(x,y)$. Hence, assuming $a_{ij}^\ast,b_{ij}^\ast$ polynomial, the model is generally written as system \eqref{eq:gamemodel}.

By the above construction, to investigate the dynamics between the players $\Gamma_1$ and $\Gamma_2$, it is sufficient to study system \eqref{eq:gamemodel} in the unit square, i.e. $(x,y)\in [0,1]^2$.

\bibliographystyle{siam}
\bibliography{Manuscript.bib}

\end{document}